\documentclass{amsart}

\usepackage{stmaryrd}
\usepackage{xcolor}
\usepackage{tikz}
\usepackage{amsmath}
\usepackage{amsthm}
\usepackage{amssymb}
\usepackage{multicol}
\usepackage{hyperref}
\usepackage{graphicx}
\usepackage{caption} 
\usepackage{subcaption} 

\numberwithin{equation}{section}
\newtheorem{theorem}{Theorem}[section]
\newtheorem{corollary}[theorem]{Corollary}
\newtheorem{example}[theorem]{Example}
\newtheorem{lemma}[theorem]{Lemma}
\newtheorem{proposition}[theorem]{Proposition}
\newtheorem{problem}{Problem}
\theoremstyle{definition}
\newtheorem{definition}[theorem]{Definition}
\newtheorem{remark}[theorem]{Remark}

\newcommand{\comments}[1]{}

\title{Transport multi-paths with capacity constraints}
\author{Qinglan Xia, Haotian Sun}
\address{
Department of Mathematics\\
University of California
at Davis\\
Davis, CA, 95616, USA }
\email{qlxia@math.ucdavis.edu, \ hatsun@ucdavis.edu}
\date{}

\subjclass[2020]{49Q22}
\keywords{branch transportation, good decomposition, 
transport multi-path, branching structure, rectifiable current, transport map,  map-compatible, simple common-source property.}

\begin{document}
\maketitle

\begin{abstract}

This article generalizes the study of branched/ramified optimal transportation to those with capacity constraints. 
Each admissible transport network studied here is represented by a transport multi-path between measures, with a capacity constraint on each of its components. 
The associated transport cost is given by the sum of the $\textbf{M}_{\alpha}$-cost of each component. 
Using this new formulation, we prove the existence of an optimal solution and provide an upper bound on the number of components for the solution. 
Additionally, we conduct analytical examinations of the properties (e.g. ``map-compatibility", and ``simple common-source property") of each solution component and explore the interplay among components, particularly in the discrete case.
\end{abstract}

\section{Introduction}
The optimal mass transportation problem aims to find an efficient transport system between sources and targets. 
The well-known Monge-Kantorovich transport problem has been extensively analyzed in recent years and found applications in many fields. Classical references can be found for instance in the books \cite{villani, villani2} by Villani, \cite{santambrogio} by Santambrogio, and the user's guide \cite{guide} by Ambrosio and Gigli. 
A variant of Monge-Kantorovich problem is the ramified optimal transportation (ROT) (also called branched transportation) problem which studies efficient transport systems with branching structures.  
The Eulerian formulation of the ROT problem was proposed by the first author in \cite{xia2003}, with related motivations, frameworks, and applications surveyed in \cite{xia2015}. 
An equivalent Lagrangian formulation of the problem was established by Maddalena, Morel, and Solimini in \cite{msm}. 
One may refer to \cite{book} for detailed discussions of the research in this direction.  
Some interesting recent developments on ROT can be found for example in \cite{bw, colombo2020, psx, ben}. 

In contrast to the Monge-Kantorovich problem where the transportation cost is solely determined by a transport map or a transport plan, the cost in the ramified transport problem is determined by the actual transport path. As illustrated in \cite{xia2003}, transport paths in ramified transportation between atomic measures can be viewed as weighted directed graphs. More precisely, let $X$ be a convex compact subset in an Euclidean space $\mathbb{R}^m$. An atomic measure on $X$ is in the form of 
$\sum_{i=1}^k m_i \delta_{x_i}$
with distinct points $x_i \in X$, and $m_i >0$ for each $i=1,\cdots,k.$
For two atomic measures 
\begin{equation}
\label{eq: atomic measures}
\textbf{a}= \sum_{i=1}^k m_i \delta_{x_i}  \ , \ \textbf{b}= \sum_{j=1}^\ell n_j \delta_{y_j}
\end{equation}
on $X$ of equal mass,
a transport path from $\textbf{a}$ to $\textbf{b}$ is a weighted directed graph $G=[V(G), E(G), w]$ consisting of a vertex set $V(G)$, a directed edge set $E(G)$ and a weight function 
$w: E(G) \rightarrow (0, +\infty)$ 
such that $\{x_1,x_2,\cdots,x_k \}\cup\{y_1,y_2,\cdots,y_\ell \} \subseteq V(G) $ and for any vertex $v\in V(G)$, there is a balance equation: 
\begin{equation}
\label{eqn: balanced_equation}
\sum_{ \substack{e\in E(G) \\  e^-=v }} w(e) \ = 
\sum_{ \substack{e\in E(G) \\  e^+=v}} w(e) \  +\  \left\{
    \begin{array}{ll}
        \ \ m_i  &\mbox{if } v=x_i \mbox{ for some } i=1,\cdots,k  \\
        -n_j     &\mbox{if }v=y_j \mbox{ for some } j=1,\cdots,\ell \\
        \ \ 0    &\mbox{otherwise}
    \end{array}
    \right.
\end{equation}
where $e^-$ and $e^+$ denote the starting and ending point of the edge $e\in E(G)$. Note that the condition (\ref{eqn: balanced_equation}) is equivalent
to \textit{Kirchhoff's circuit law},
or the requirement that mass is conserved at every interior vertex.

For any real number $\alpha \in [0,1]$, the $\mathbf{M}_\alpha$ cost  of 
$G =[V(G),E(G), w] $ 
is defined by
\begin{equation}
\label{eq: ramified transport cost function}
\textbf{M}_\alpha(G):=\sum_{e\in E(G)} \left(w(e)\right)^\alpha \mathcal{H}^1(e),
\end{equation}
where
$\mathcal{H}^1(e)$ is the
1-dimensional Hausdorff measure
or length of the edge $e$.

As an extension to a more general setting, a transport path between two Radon measures $\mu^{\pm}$ of equal mass can be viewed as a rectifiable $1$-current $T$ with $\partial T=\mu^+-\mu^-$, which will be discussed in detail in the preliminary section. The corresponding transport cost on $T$ is its $\alpha$-mass $\mathbf{M}_{\alpha}(T)$. Under these notations, the ramified optimal transport problem is: 
\[\label{ROT}
\tag{\textbf{ROT}} 
\text{Minimize }\mathbf{M}_{\alpha}(T)
\text{ among all } T\in Path(\mu^-, \mu^+),\]
where $Path(\mu^-, \mu^+)$ is the collection of all transport paths from $\mu^-$ to $\mu^+$. An $\mathbf{M}_\alpha$-minimizer in $Path(\mu^-, \mu^+)$ is called an $\alpha$-\textit{optimal transport path} from $\mu^-$ to $\mu^+$.

This article aims to study the impact of capacity constraints on optimal transport paths.
Regardless of whether in atomic case or general case, the amount of mass that can be transported via an admissible transport path has no restrictions. 
Hence, the phenomenon of first aggregating the total mass from the source into one place and then transporting it through a single curve is permitted and prevalent in ramified transport paths.

As opposed to the theoretically permitted unlimited aggregation of mass,
transportation in reality often takes place through various kinds of mediums, which have transport capacity limiting the maximum amount of mass they can carry all at once. 
For instance, roads have limited number of lanes for cars, and cars have limited seats for passengers.
This brings naturally the question of ramified transport paths with capacity constraints, which can be roughly described by imposing an upper bound (called the capacity) on the weight function of a weighted directed graph or on the density function of a rectifiable $1$-current. 
This motivates us to consider the following ramified transport problem in the discrete case:
Given two atomic measures $\mathbf{a}$, $\mathbf{b}$ on $X$ with equal mass, and capacity $c>0$, 
\[\label{problem 0}
\tag{Problem 0}
\text{ minimize } \mathbf{M}_\alpha(G) \text{ among all } G \in Path(\mathbf{a},\mathbf{b}) \text{  with }  w(e)\le c, \text{ for all } e\in E(G).  
\]

From the description of this problem, if we assume the total mass $\|\mathbf{a}\|=\|\mathbf{b}\| \le c$, this is equivalent to imposing no restriction on the transport capacity of $G$.  
Note that after imposing the capacity constraint, one obvious observation is that a previously well-defined transport path $G\in Path(\mathbf{a},\mathbf{b})$, which has no capacity constraints, is not necessarily an admissible transport path anymore. 
This can be demonstrated in the following example:

\begin{example}

\begin{figure}[h]
\centering
\begin{subfigure} {0.45\textwidth}
\centering
\begin{tikzpicture}[>=latex]
\filldraw[black](0,0) circle (1pt) node[anchor=north]{$x_3$};
\filldraw[black](-2,3) circle (1pt) node[anchor=south]{$x_1$};
\filldraw[black](2,3) circle (1pt) node[anchor=south]{$x_2$};
\filldraw[black](0,1.5) circle (1pt) node[anchor=east]{$x_4$};
\draw[->] (-2,3)--(0,1.5); 
\draw[->] (2,3)--(0,1.5); 
\draw[thick,->] (0,1.5)--(0,0);
\filldraw[black](0,-0.5) circle (0pt) node[anchor=north]{ };
\filldraw[black](-1,3) circle (0pt) node[anchor=north]{$\frac{1}{2}$};
\filldraw[black]( 1,3) circle (0pt) node[anchor=north]{$\frac{1}{2}$};
\filldraw[black]( 0,0.7) circle (0pt) node[anchor=east]{$1$};
\end{tikzpicture}
\caption{Y-shaped} \label{subfig: Y shaped}
\end{subfigure}
\begin{subfigure}{0.45\textwidth}
\centering
\begin{tikzpicture}[>=latex]
\filldraw[black](6,0) circle (1pt) node[anchor=north]{$x_3$};
\filldraw[black](4,3) circle (1pt) node[anchor=south]{$x_1$};
\filldraw[black](8,3) circle (1pt) node[anchor=south]{$x_2$};
\filldraw[black](5.25,1.75) circle (1pt) node[anchor=east]{$x_5$};
\draw[->] (4,3)--(5.25,1.75); 
\draw[->] (8,3)--(5.25,1.75); 
\draw[thick,->] (5.25,1.75)--(6,0);
\draw[dashed,->] (8,3) -- (6,0);
\filldraw[black](6,-0.5) circle (0pt) node[anchor=north]{ };
\filldraw[black](4.75,3) circle (0pt) node[anchor=north]{$\frac{1}{2}$};
\filldraw[black](6.5,3) circle (0pt) node[anchor=north]{$\frac{1}{6}$};
\filldraw[black](7,1.5) circle (0pt) node[anchor=north]{$\frac{1}{3}$};
\filldraw[black](5.25,1.5) circle (0pt) node[anchor=north]{$\frac{2}{3}$};
\end{tikzpicture}
\caption{Mixture of Y-shaped and V-shaped}
\label{subfig: mix Y and V}
\end{subfigure}    
\caption{Y-shaped \& Mixture of Y-shaped and V-shaped.}
\label{fig: Y and mix of Y and V}
\end{figure}
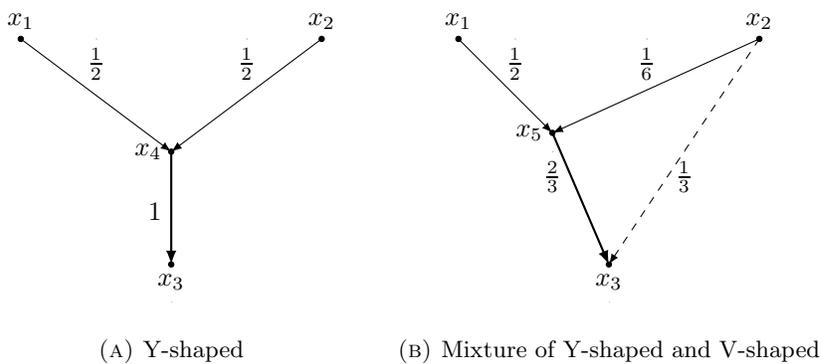

As shown in Figure \ref{fig: Y and mix of Y and V}, suppose we want to transport mass from 
$$\mathbf{a}=\frac{1}{2}\delta_{x_1}+\frac{1}{2}\delta_{x_2} \text{ to }  \mathbf{b}=\delta_{x_3},$$
with an upper bound $c=\frac{2}{3}$ imposed on weight functions. 
In this case,  a ``Y-shaped" transport path no longer satisfies the restriction on weight functions, since after merging at $x_4$ the mass will reach 1. 
Changing to another kind of branching structure which is a mixture of ``V-shaped" and ``Y-shaped" will resolve this issue. 
One of the possible cases is merging $\frac{1}{2}$ from $x_1$ and $\frac{1}{6}$ from $x_2$, and let the remaining $\frac{1}{3}$ from $x_2$ transport directly through the dash line. 
\end{example}

Another observation is about the non-compactness of the family of admissible transport paths in \ref{problem 0}, due to the ``merging" effect shown in Example \ref{ex: aggregation}. 
As a result, \ref{problem 0} may fail to have a solution.

\begin{example} \label{ex: aggregation}
Let $\textbf{a}=\delta_{x}$, $\textbf{b}=\delta_{y}$ for some $x,y \in \mathbb{R}^2$, and
suppose the transport capacity $c=1/n$ for some integer $n\ge 2$. 
As shown in Figure \ref{fig:convergence_fail}, we may construct a sequence of transport paths in $Path(\textbf{a}, \textbf{b})$ satisfying the capacity constraint, that converges to the straight line segment $\llbracket x, y \rrbracket$ with weight $1$, which does not satisfy the capacity constraint anymore.

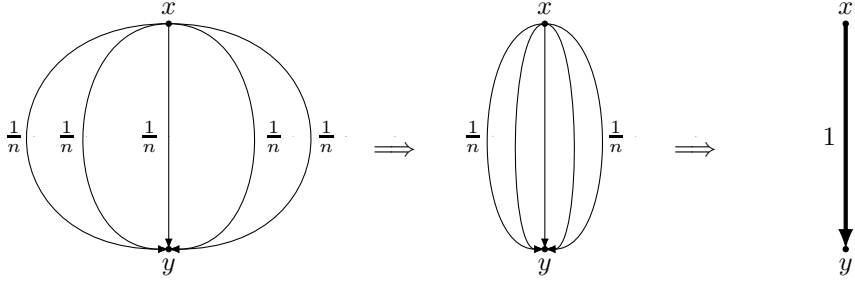
\begin{figure}[h]
\centering
\begin{tikzpicture}[>=latex]
\filldraw[black](0,0) circle (1pt) node[anchor=north]{$y$};
\filldraw[black](0,3) circle (1pt) node[anchor=south]{$x$};

\draw[->](0,3) -- (0,0);
\draw[->](0,3) .. controls (-2.5,3) and (-2.5,0) .. (0,0);
\draw[->](0,3) .. controls (-1.5,3) and (-1.5,0) .. (0,0);
\draw[->](0,3) .. controls (1.5,3) and (1.5,0) .. (0,0);
\draw[->](0,3) .. controls (2.5,3) and (2.5,0) .. (0,0);

\filldraw[black](0,1.5) circle (0pt) node[anchor = east]{$\frac{1}{n}$};
\filldraw[black](-1.1,1.5) circle (0pt) node[anchor = east]{$\frac{1}{n}$};
\filldraw[black](-1.8,1.5) circle (0pt) node[anchor = east]{$\frac{1}{n}$};
\filldraw[black](1.65,1.5) circle (0pt) node[anchor = east]{$\frac{1}{n}$};
\filldraw[black](2.35,1.5) circle (0pt) node[anchor = east]{$\frac{1}{n}$};

\filldraw[black](3,1.5) circle (0pt) node[anchor=north]{$\Longrightarrow$};

\filldraw[black](5,0) circle (1pt) node[anchor=north]{$y$};
\filldraw[black](5,3) circle (1pt) node[anchor=south]{$x$};
\draw[->](5,3) -- (5,0);
\draw[->](5,3) .. controls (4,3) and (4,0) .. (5,0);
\draw[->](5,3) .. controls (4.5,3) and (4.5,0) .. (5,0);
\draw[->](5,3) .. controls (5.5,3) and (5.5,0) .. (5,0);
\draw[->](5,3) .. controls (6,3) and (6,0) .. (5,0);

\filldraw[black](6.2,1.5) circle (0pt) node[anchor = east]{$\frac{1}{n}$};
\filldraw[black](4.3,1.5) circle (0pt) node[anchor = east]{$\frac{1}{n}$};

\filldraw[black](7,1.5) circle (0pt) node[anchor=north]{$\Longrightarrow$};

\filldraw[black](9,0) circle (1pt) node[anchor=north]{$y$};
\filldraw[black](9,3) circle (1pt) node[anchor=south]{$x$};
\draw[ultra thick,->](9,3) -- (9,0);
\filldraw[black](9,1.5) circle (0pt) node[anchor = east]{$1$};

\filldraw[black](10,0) circle (0pt) node[anchor=north]{ };
\end{tikzpicture}    
    \caption{The above pictures give an illustration of ``convergence" when $n=5$.}
    \label{fig:convergence_fail}
\end{figure}
\end{example}

To overcome this non-compactness issue, in the following, we modify \ref{problem 0} by expressing each transport network as a transport ``multi-path" $\vec{T}=(T_1,T_2,\cdots,T_k,\cdots)$ such that each component $T_k$ satisfies the given capacity constraints.

\begin{problem}[\textbf{Ramified transportation with capacity}]
\label{problem: main}
Let $\mu^-$, $\mu^+$ be two Radon measures on $ \mathbb{R}^m$ with equal mass, supported on compact sets, $\alpha \in 
(0,1)$, and $c>0$. 
Minimize  $$\mathbf{M}_\alpha(\vec{T}):=\sum_{k=1}^\infty \mathbf{M}_\alpha(T_k)$$
among all
$\vec{T}=(T_1,T_2,\cdots,T_k,\cdots)$ such that for each $k$,
\begin{equation}\label{eqn: problem_1}
T_k \in Path(\mu^-_k, \mu^+_k), \   
\sum_{k=1}^\infty \mu_k^- = \mu^-, \ \sum_{k=1}^\infty \mu_k^+ = \mu^+, and\ 0<\| \mu_k^-\| = \| \mu_k^+\| \le c.
\end{equation}
\end{problem}
Each $\vec{T}=(T_1,T_2,\cdots,T_k,\cdots)$
satisfying (\ref{eqn: problem_1}) is called a {\it transport multi-path} from $\mu^-$ to $\mu^+$ with capacity $c$. 
The family of all such transport multi-paths is denoted by $Path_c(\mu^-,\mu^+)$.

The article is organized as follows. 
We first review in $\S 2$ some related concepts in geometric measure theory and optimal transportation. In particular, the good decomposition of normal 1-currents and the map-compatibility of transport paths.  
In $\S 3$, we show in Theorem \ref{thm: existence of optimal transport paths} that Problem \ref{problem: main} indeed has a solution $\vec{T}^*=(T_1,T_2,\cdots, T_N)$, and the number of its components $N$  is bounded above by $2\|\mu^-\|/c + 1$. 

After achieving the existence result, we also study properties of the optimal multi-path $T^*$ in $\S 4$.
Note that each component $T_k$ of $T^*$ is automatically an optimal transport path from $\mu_k^-$ to $\mu_k^+$. Previously, in  \cite[Section 2.1]{xia2015}, we listed some basic properties of an optimal transport path such as the acyclic property, balance equations at each interior vertices, uniform upper bounds on the degree of vertices, and uniform lower bounds on angles between edges. Moreover, in \cite{decomposition}, we showed that each optimal transport path has a ``simple common-source" property.\footnote{A transport path $T$ satisfies the ``simple common-source" property if there exists a ``good decomposition" $\eta$ of $T$ such that the pairwise intersection of preimages of the target points under $\eta$ contains at most one point.} All these properties now automatically hold for each component $T_k$. 

To get a better description of the optimal multi-path $T^*$, we study some other properties of its components $ T_k$ in $\S 4$, that do not hold for a general optimal transport path. 
When both $\mu^-$ and $\mu^+$ are atomic measures, we first show in Theorem \ref{thm: X_j's mutually disjoint} that each $T_k$ is a map-compatible transport path with at most $N_2-1$ exceptions, where $N_2$ is the cardinality of the target measure $\mu^+$. 
Moreover, in Theorem \ref{General source overlap property}, we further generalize the simple common-source property of an optimal transport path to an analogous result for components of an optimal multi-path. In the end, we provide a detailed analysis of the single target case in $\S 5$.

\section{Preliminaries}

We first recall some related terminologies from geometric measure theory  and optimal transportation theory.

\subsection{Concepts in geometric measure theory}
\subsubsection{Currents (See \cite{Lin, Simon})}
For any an open set $U$ in $\mathbb{R}^{m}$
and $k \le m$, let $\mathcal{D}^k(U)$ be the set of all $C^\infty$ $k$-forms in $U$ with compact supports. The space $\mathcal{D}_k(U)$ of \textit{$k$-currents} 
is the dual space of $\mathcal{D}^k(U)$.

For any current $T \in \mathcal{D}_k(U)$, the \textit{mass} of $T$ is defined by
$$
\mathbf{M}(T)= \sup \{ T(\omega) \,:\,  \|\omega\| \le 1, \omega \in \mathcal{D}^k(U)\}, $$
where $$ \|\omega \| := \sup_{x\in U} \sqrt{\omega(x) \cdot \omega(x)}, \ \forall \omega \in \mathcal{D}^k(U).
$$
Also, its \textit{boundary} $\partial T \in \mathcal{D}_{k-1}(U)$ is defined by 
$$\partial T(\omega) := T(d\omega), \forall  \omega \in \mathcal{D}^{k-1}(U), \text{ when } k\ge 1 ,$$
and $\partial T :=0$ when $k=0$.

A set $M \subset \mathbb{R}^{m}$ is said to be countably \textit{$k$-rectifiable} if 
$$
M \subseteq M_0 \cup \left( \bigcup_{j=1}^\infty F_j(\mathbb{R}^k)\right),
$$
where $\mathcal{H}^k(M_0)=0$  under the $k$-dimensional Hausdorff measure $\mathcal{H}^k$, and $F_j: \mathbb{R}^k \to \mathbb{R}^{m}$ are Lipschitz functions for $j=1,2,\cdots.$ 
For any $k$-current $T\in \mathcal{D}_k(U)$, we say that $T$ is \textit{rectifiable} if for each $\omega \in \mathcal{D}^k(U)$,
$$T(\omega) = \int_M \langle \omega(x) , \xi(x) \rangle \theta(x) \,d\mathcal{H}^k(x),$$
where $M$ is an $\mathcal{H}^k$-measurable countably $k$-rectifiable subset of $U$, $\theta(x)$ is a locally $\mathcal{H}^k$-integrable positive function, called multiplicity, and $\xi$ is an orientation of $M$ in the sense that $\xi(x)$ is a simple unit $k$-vector that spans the approximate tangent space $T_xM$ for $\mathcal{H}^k$-almost every $x\in M$. We will denote such a $T$ by $\underline{\underline{\tau}}(M,\theta,\xi)$. When $T=\underline{\underline{\tau}}(M,\theta,\xi)$ is a rectifiable $k$-current, its mass
$$
\mathbf{M}(T) = \int_M \theta(x) \,d\mathcal{H}^k(x).
$$

A current $T\in \mathcal{D}_k(U)$ is said to be \textit{normal} if 
$\mathbf{M}(T) + \mathbf{M}(\partial T) < \infty$.
In \cite{Paolini}, Paolini and Stepanov introduced the concept of subcurrents: For any $T, S\in \mathcal{D}_k(U)$, $S$ is called a \textit{subcurrent} of $T$ if
$$ \mathbf{M}(T-S) + \mathbf{M}(S) = \mathbf{M}(T).$$ A normal current $T\in \mathcal{D}_k(\mathbb{R}^m)$ is \textit{acyclic} if there is no non-trivial subcurrent $S$ of $T$ such that $\partial S =0$. 

\subsubsection{Good decomposition}
In \cite{smirnov}, Smirnov showed that every acyclic normal $1$-current can be written as the weighted average of simple Lipschitz curves in the following sense.

Let $\Gamma$ be the space of $1$-Lipschitz curves $\gamma: [0,\infty) \to \mathbb{R}^m$, which are eventually constant. 
For $\gamma \in \Gamma$, we denote 
$$t_0(\gamma):=\sup \{t: \gamma \mbox{ is constant on } [0,t]  \}, \ 
t_\infty(\gamma):=\inf \{t: \gamma \mbox{ is constant on } [t,\infty)  \} ,$$
and $p_0(\gamma):= \gamma(0)$, $p_\infty(\gamma):= \gamma(\infty)=\lim_{t\to \infty} \gamma(t)$. A curve $\gamma \in \Gamma$ is simple if $\gamma(s)\not=\gamma(t)$ for every $t_0(\gamma) \le s < t \le t_\infty(\gamma)$. Also, let $Im(\gamma)$ denote the image of $\gamma$ in $\mathbb{R}^m$.
In the following contexts, we adopt the notations: for any points $x,y \in \mathbb{R}^m$ and subset $A\subseteq \mathbb{R}^m$, denote
\begin{align}
    \Gamma_x & = \{\gamma \in \Gamma : x \in \mathrm{Im}(\gamma) \}, \label{gamma_{x}}\\
    \Gamma_{x,y}&=\{\gamma\in \Gamma: p_0(\gamma)=x, \ p_\infty(\gamma)=y\}, \label{gamma_{x,y}}
    \\
    \Gamma_{A,y}&=\{\gamma\in \Gamma: p_0(\gamma)\in A, \ p_\infty(\gamma)=y\}. \label{gamma_{A,y}}
\end{align}

For each simple curve $\gamma \in \Gamma$, we may associate it with the following rectifiable $1$-current, 
\begin{equation}
\label{eqn: I_gamma}
I_\gamma:= \underline{\underline{\tau}}\left( \mathrm{Im}(\gamma), \frac{\gamma'}{|\gamma'|}, 1  \right).
\end{equation}

\begin{definition} (see \cite{colombo}, \cite{colombo2020}, \cite{smirnov})
\label{defn: good decomposition}
Let $T$ be a normal 1-current in $\mathbb{R}^m$ and let
$\eta$ be a finite positive measure on $\Gamma$
such that
\begin{equation}
\label{eqn: T_eta}
    T=\int_\Gamma I_\gamma \,d\eta(\gamma)
\end{equation}
in the sense that 
\begin{equation}
    T(\omega)=\int_\Gamma I_\gamma(\omega) \,d\eta(\gamma), \forall \omega\in \mathcal{D}^1(\mathbb{R}^m).
\end{equation}

We say that $\eta$ is a \textit{good decomposition} of $T$  if $\eta$ is supported on non-constant, simple curves and
satisfies the following equalities:
\begin{itemize}
\item[(a)] $\mathbf{M}(T)=\int_\Gamma \mathbf{M}(I_\gamma) d\eta(\gamma)=\int_\Gamma \mathcal{H}^1(Im(\gamma)) d\eta(\gamma)$;

\item[(b)] $\mathbf{M}(\partial T)=\int_\Gamma \mathbf{M}(\partial I_\gamma) d\eta(\gamma)=2 \eta(\Gamma)$.
\end{itemize}
\end{definition}

Due to \cite{smirnov} and \cite[Theorem 5.1]{Paolini}, it turns out that every acyclic normal $1$-current $T$ has a good decomposition. When $\eta$ is a good decomposition of $T$, as collected in \cite[Proposition 3.6]{colombo}, it follows that
\begin{itemize}
    \item[(a)]
    \begin{equation} \label{eqn: startingending measure}
    \mu^- = \int_\Gamma \delta_{\gamma (0)} \, d\eta(\gamma),\ 
    \mu^+ = \int_\Gamma \delta_{\gamma (\infty)} \, d\eta(\gamma).
    \end{equation}
    
    \item [(b)]
    If $T = \underline{\underline{\tau}}(M,\theta,\xi)$ is rectifiable, then
    \begin{equation}
    \label{eqn: theta(x)}
    \theta(x) = \eta (\{\gamma \in \Gamma : x \in \mathrm{Im}(\gamma) \} )
    \end{equation}
     for $\mathcal{H}^1$-a.e. $x \in M.$

    \item[(c)] For every $\tilde{\eta} \leq  \eta$, the representation
    \[\tilde{T}=\int_\Gamma I_\gamma d\tilde{\eta}(\gamma)\]
    is a good decomposition of $\tilde{T}$. Moreover, if $T=\underline
    {\underline {\tau} }\left(M, \theta, \xi\right)$ is rectifiable, then $\tilde{T}$ can be written as $\tilde{T}=\underline
    {\underline {\tau} }(M, \tilde{\theta}, \xi )$ with
    \begin{equation}\label{eqn: tilde_theta_x}
         \tilde{\theta}(x)\le \min\{\theta(x), \tilde{\eta}(\Gamma)\} 
    \end{equation}
    for $\mathcal{H}^1$-a.e. $x\in M$.
     
\end{itemize}

\subsection{Concepts in optimal transportation theory}
Now, we recall some related concepts from  Monge-Kantorovich transportation (\cite{Luigi}), and the ramified transportation (\cite{xia2003, xia2015}). 
\subsubsection{Optimal transportation problems}
Let $X$ be a convex compact subset of $\mathbb{R}^m$, and let $\mu^-$ (source) and $\mu^+$ (target) be two Radon measures supported on $X$ of equal mass.

\begin{itemize}
\item[(a)] In the Monge formulation of transportation, a transport map from $\mu^-$ to $\mu^+$ is a Borel map $\varphi: X \to X$ with $\varphi_\# \mu^- = \mu^+$. 
Denote $$Map(\mu^-,\mu^+)$$ as the set of transport maps from $\mu^-$ to $\mu^+$.
Given a non-negative Borel function $c(x,y)$ on $X \times X$, the Monge optimal transport problem is: Minimize
$$I_c(\varphi) := \int_X c(x,\varphi(x)) \,d\mu^-$$
among all $\varphi \in Map(\mu^-,\mu^+)$.

\item[(b)] In the Kantorovich formulation of transportation, a transport plan from $\mu^-$ to $\mu^+$ is a measure $\pi$ on the product space $X \times X$ with marginals $\mu^- $ and $\mu^+$.
Denote $$Plan(\mu^-,\mu^+)$$ as the set of transport plans from $\mu^-$ to $\mu^+$.
Given a cost function $c(x,y)$ on $X \times X$, the Kantorovich optimal transport problem is: Minimize
$$J_c(\pi) := \int_{X \times X}  c(x,y) \,d\pi(x,y)$$
among all $\pi \in Plan(\mu^-,\mu^+)$.
\item[(c)] In the ramified transportation, a transport path from $\mu^-$ to $\mu^+$ is a rectifiable $1$-current $T$ such that $\partial T = \mu^+ - \mu^-$.
Denote $$Path(\mu^-,\mu^+)$$ as the set of all transport paths from $\mu^-$ to $\mu^+$.
Given $\alpha \in [0,1)$, the ramified optimal transport problem is: Minimize
$$\mathbf{M}_\alpha(T) := \int_M \theta(x)^\alpha \,d\mathcal{H}^1$$
among all $T=\underline{\underline{\tau}}(M,\theta,\xi)\in Path(\mu^-,\mu^+)$.
\end{itemize}

\subsubsection{Compatibility and the simple common-source property}
In \cite{xia2003,decomposition}, we studied a special kind of transport paths that are compatible with a transport map (or a transport plan).

\begin{definition}(\cite[Definition 5.3]{decomposition})
\label{def: Radon_compatibility}
Let $\mu$ and $\nu$ be two Radon measures on $X$ of equal mass and $T \in Path (\mu,\nu)$.  
\begin{itemize}
    \item[(a)] For any $\pi \in Plan (\mu,\nu)$, if there exists a finite Borel measure $\eta$ on $\Gamma$ such that 
$$T = \int_\Gamma  I_\gamma d\eta, \text{ and } \pi = \int_\Gamma \delta_{ ( p_0(\gamma), p_\infty(\gamma) ) } d\eta,$$
then we say the pair $(T, \pi) $ is compatible (with respect to $\eta$).
\item[(b)] For any $\varphi \in Map (\mu,\nu)$, we say the pair $(T,\varphi)$ is compatible if $(T, \pi_\varphi)$ is compatible, where 
$\pi_\varphi = (id \times \varphi)_{\#}\mu$.
\item[(c)] $T \in Path (\mu,\nu)$ is called map-compatible if there exists a map $\varphi \in Map (\mu,\nu)$ such that  $(T,\varphi)$ is compatible.
\end{itemize}
\end{definition}

We now characterize map-compatible transport paths in the case of atomic measures. 
Let 
\begin{equation}\label{eqn: measures}
\mu^- = \sum_{i=1}^{N_1} m'_i \delta_{x_i} \text{ and }\mu^+ = \sum_{j=1}^{N_2} m_j \delta_{y_j} \text{ with } \sum_{i=1}^{N_1} m'_i = \sum_{j=1}^{N_2} m_j < \infty
\end{equation}
be two atomic measures on $X$ with $N_1, N_2 \in \mathbb{N} \cup \{\infty\}$.  For each $i=1,2,\cdots, N_1$, $j=1,2,\cdots, N_2$, as given in (\ref{eqn: measures}), let $\Gamma_{x_i, y_j}$ denote the set of all $1$-Lipschitz curves in $\Gamma$ from $x_i$ to $y_j$. 
For any finite Borel measure $\eta$ of $T$, and each $y_j \in \{y_1,y_2,\cdots,y_{N_2}\}$, denote
\begin{equation}
 \label{eqn: X_j}   
X_j(\eta) := \{x_i\in X : \eta(\Gamma_{x_i,y_j})>0\}.
\end{equation}

Using this notation, we now characterize map-compatible transport paths in $Path(\mu^-, \mu^+)$ as follows:
\begin{proposition}
\label{prop: X_1 intersect X_2 = empty gives transport map}
Let $\mu^-$ and $\mu^+$ be two atomic measures of equal mass as given in (\ref{eqn: measures}), and $T\in Path(\mu^-,\mu^+)$.

\begin{itemize}

\item[(1)] If $T$ is map-compatible with respect to a finite Borel measure $\eta$,
then $$|X_{j_1}(\eta) \cap X_{j_2}(\eta)|=0 \text{ for any } 1 \le j_1 < j_2 \le N_2.$$

\item[(2)] If $T$ has a good decomposition $\eta$ such that $|X_{j_1}(\eta) \cap X_{j_2}(\eta)|=0$ for any $1 \le j_1 < j_2 \le N_2$, then $T$ is map-compatible.
\end{itemize}
\end{proposition}

\begin{proof}
(1) Suppose $\varphi \in Map(\mu^-,\mu^+)$ and $(T, \varphi)$ is compatible.
By Definition \ref{def: Radon_compatibility}, there exists a finite Borel measure $\eta$ on $\Gamma$ such that 
$$T = \int_\Gamma  I_\gamma d\eta, \   \pi_\varphi = \int_\Gamma \delta_{ ( p_0(\gamma), p_\infty(\gamma) ) } d\eta ,$$
where $\pi_\varphi = (id \times \varphi)_\# \mu^-$.

Since $\mu^-,\mu^+$ are atomic measures as stated in (\ref{eqn: measures}), by using the proof given in \cite[Proposition 5.4]{decomposition}, it follows that
\begin{equation}
\label{eq: eta format}
\pi_\varphi = \sum_{i,j} \eta (\Gamma_{x_i,y_j}) \delta_{(x_i,y_j)}.    
\end{equation}
Also, direct calculation gives
$$\pi_\varphi (\{x_i\} \times \{y_j\}) = (id \times \varphi)_\# \mu^- (\{x_i\} \times \{y_j\}) = \mu^- (\{x_i\} \cap \varphi^{-1}(\{y_j\})),$$
so that $\pi_\varphi (\{x_i\} \times \{y_j\})= \mu^- (\emptyset) =0$ when $\varphi(x_i)\not=y_j$, and $\pi_\varphi (\{x_i\} \times \{y_j\})= \mu^- (\{x_i\}) >0$ when $\varphi(x_i)=y_j$.
Using equation (\ref{eq: eta format}), we have 
$$\eta (\Gamma_{x_i,y_j})=0  \text{ when }  \varphi(x_i)\not=y_j , \text{ and } \eta (\Gamma_{x_i,y_j})>0 \text{ when } \varphi(x_i)=y_j.$$
Hence, for each $1 \le j \le N_2$, $$ X_j(\eta) := \{x_i\in X : \eta (\Gamma_{x_i,y_j}) >0 \} = \{x_i \in X : \varphi(x_i)=y_j\} = \varphi^{-1}(\{y_j\}).$$
Since $\varphi$ is a map from $\{x_1,x_2,\cdots,x_{N_1}\}$ to $\{y_1,y_2,\cdots,y_{N_2}\}$, for $1 \le j_1 < j_2 \le N_2$ we have  
$$X_{j_1}(\eta) \cap X_{j_2}(\eta) = \varphi^{-1}(\{y_{j_1}\}) \cap \varphi^{-1}(\{y_{j_2}\}) = \emptyset.$$

(2) Now, suppose there exists a good decomposition $\eta$ of $T$ such that
$|X_{j_1}(\eta) \cap X_{j_2}(\eta)|=0$ for any $1 \le j_1 < j_2 \le N_2$. In particular,
\begin{equation}
\label{eqn: T}
    T=\int_\Gamma I_\gamma d\eta.
\end{equation}
In this case,  we may define 
$$\varphi : \bigcup_{j=1}^{N_2}  X_j (\eta)=\{x_1,x_2,\cdots, x_{N_1}\}  \to \{y_1,y_2,\cdots,y_{N_2}\}, $$
by setting $ \varphi(x):=y_j $  for  $x\in X_j(\eta)$. Note that $\varphi_\# \mu^- = \mu^+$ because by direct calculation,  
$$\varphi_\# \mu^- (\{y_j \}) =  
\mu^- (X_j(\eta)) =  
\sum_{i=1}^{N_1} \eta(\Gamma_{x_i,y_j}) =
\mu^+(\{y_j\}),
$$
for each $j=1,2,\cdots, N_2$. 
Moreover,
\begin{eqnarray*}
(id \times \varphi)_\# \mu^- 
&=& 
(id \times \varphi)_\# 
\left(\sum_{i=1}^{N_1} \left(\sum_{j=1}^{N_2}  \eta(\Gamma_{x_i,y_j})\right)\delta_{x_i}\right) \\
&=&
\sum_{i=1}^{N_1} \sum_{j=1}^{N_2}  \eta(\Gamma_{x_i,y_j}) (id \times \varphi)_\#\delta_{x_i} \\
&=& 
\sum_{j=1}^{N_2}
\left(
\sum_{i=1}^{N_1}   \eta(\Gamma_{x_i,y_j}) (id \times \varphi)_\#\delta_{x_i} \right)\\
&=&
\sum_{j=1}^{N_2} \left(\sum_{x_i \in X_j(\eta)} \eta(\Gamma_{x_i,y_j}) \delta_{(x_i,y_j)}\right)  \\
&=&
\sum_{j=1}^{N_2} \left( \sum_{x_i \in X_j(\eta)} \int_{\Gamma_{x_i,y_j} }  \delta_{ (p_0(\gamma)  , p_\infty(\gamma) )}\, d\eta \right) \\
&=&
\sum_{j=1}^{N_2} \sum_{i =1}^{N_1} \int_{\Gamma_{x_i,y_j} }  \delta_{ (p_0(\gamma)  , p_\infty(\gamma) )}\, d\eta=\int_{\Gamma}\delta_{ (p_0(\gamma)  , p_\infty(\gamma) )}\, d\eta, 
\end{eqnarray*}
which implies $(T, \varphi)$ is compatible by (\ref{eqn: T}).

\end{proof}

Motivated by Proposition \ref{prop: X_1 intersect X_2 = empty gives transport map}, we are also interested in transport path $T$ that has a good decomposition $\eta$ such that
\begin{equation}
\label{eqn: X_j_intersection}
    |X_{j_1}(\eta)\cap X_{j_2}(\eta)|\le 1, \text{ for each pair } 1\le j_1<j_2\le N.
\end{equation}
The inequality (\ref{eqn: X_j_intersection}) says that the intersection $X_{j_1}(\eta)\cap X_{j_2}(\eta)$ is either empty or a singleton. In other words, every two target points can have at most one common source point. 

A transport path $T$ is said to have the {\it simple common-source} property if there exists a good decomposition $\eta$ of $T$ that satisfies the inequality (\ref{eqn: X_j_intersection}).
We claim that each optimal transport path has this property. To see this, we first recall the definition of cycle-free paths:

\begin{definition} (\cite[Definition 4.2]{decomposition})
\label{def: cycle-free current}
Let $T=\underline{\underline{\tau}}(M,\theta,\xi)$ and $S=\underline{\underline{\tau}}(N,\phi,\zeta)$ be two real rectifiable $k$-currents.
\begin{itemize}
    \item[(a)] We say $S$ is
{\it on} $T$ if $\mathcal{H}^k(N\setminus M)=0$, and $\phi(x)\le \theta(x)$ for $\mathcal{H}^k$ almost every $ x\in N$. 
\item[(b)] $S$ is called a {\it cycle} on $T$ if $S$ is on $T$ and $\partial S=0$.
\item[(c)] $T$ is called {\it cycle-free} if except for the zero current, there is no other cycle on $T$.
\end{itemize} 
\end{definition}

As stated in \cite[Remark 4.3]{decomposition}, the concept of ``cycle-free" is different from ``acyclic".  A cycle-free current is automatically acyclic, but not vice versa.  Also, by \cite[Corollary 4.5]{decomposition}, each optimal transport path is automatically cycle-free.

In \cite[Propositions 4.6--4.7]{decomposition}, we showed that
\begin{proposition}
\label{prop: single component less than 1 point}
Each cycle-free transport path $T\in Path(\mu^-, \mu^+)$ has a good decomposition $\eta$ that satisfies (\ref{eqn: X_j_intersection}).
\end{proposition}

As a result, each cycle-free transport path, in particular each optimal transport path, has the simple common-source property.

\section{Existence of an optimal transport multi-path}

In this section, we will show that Problem \ref{problem: main} has a solution. 
We first note that, by (\ref{eqn: problem_1}), for any multi-path $\vec{T}=\left(T_1,T_2,\cdots, T_k, \cdots\right)\in Path_c(\mu^-, \mu^+)$, the series
$$\sum_{k=1}^\infty T_k \in Path(\mu^-,\mu^+)$$
provided that the series is convergent as currents, i.e.,  for any $\omega \in \mathcal{D}^1 (\mathbb{R}^m)$, the series $\sum_{k=1}^\infty T_k(\omega)$ of real numbers converges.

\begin{lemma} 
\label{lem: M alpha cost triagle inequality}
For any convergent series $ \sum_{i=1}^\infty T_i$ of rectifiable 1-currents, if $0\le \alpha \le 1$ then
$$\mathbf{M}_\alpha\left(\sum_{k=1}^\infty T_k\right) \le \sum_{k=1}^\infty \mathbf{M}_\alpha(T_k). $$
\end{lemma}

\begin{proof} 
Suppose $T_k = \underline{\underline{\tau}}(M_k,\theta_k,\xi_k ) $, and let $\omega \in \mathcal{D}^1 (\mathbb{R}^m)$, then
\[
T_k(\omega) = \int_{M_k} \langle \omega(x), \xi_k(x)\rangle \theta_k(x) \, d\mathcal{H}^1(x),
\] 
and 
\begin{eqnarray*}
\sum_{k=1}^\infty T_k (\omega) 
&=& 
\sum_{k=1}^\infty \int_{M_k} \left\langle \omega(x), \xi_k(x) \right\rangle  \theta_k(x) \,d\mathcal{H}^1(x) \\
&=&
\int_{\bigcup\limits_{k=1}^\infty M_k} \left\langle \omega(x), \sum_{k=1}^\infty \xi_k(x)\theta_k (x)  \right\rangle  \,d\mathcal{H}^1(x) . 
\end{eqnarray*}
Here, we adopt the convention that for each $k$, $\theta_k(x)=0$ when $x \not\in M_k$.  
Since $0\le \alpha \le 1$, then for each $n \in \mathbb{N}$,
$$\left( \sum_{k=1}^n \theta_k(x) \right)^\alpha \le \sum_{k=1}^n \theta_k(x)^\alpha \le \sum_{k=1}^\infty \theta_k(x)^\alpha ,$$
so that 
$$\left( \sum_{k=1}^\infty \theta_k(x) \right)^\alpha = 
\lim_{n \to \infty} \left( \sum_{k=1}^n \theta_k(x) \right)^\alpha \le 
\sum_{k=1}^\infty \theta_k(x)^\alpha .$$
Therefore,
\begin{eqnarray*}
\mathbf{M}_\alpha \left( \sum_{k=1}^\infty T_k \right) 
&\le & 
\int_{\bigcup\limits_{k=1}^\infty M_k} \left(\sum_{k=1}^\infty\theta_k (x)\right)^\alpha d\mathcal{H}^1(x)  \\
&\le& 
\int_{\bigcup\limits_{k=1}^\infty M_k} \sum_{k=1}^\infty \theta_k(x)^\alpha d\mathcal{H}^1(x)  
=
\sum_{k=1}^\infty \int_{\bigcup\limits_{k=1}^\infty M_k} \theta_k(x)^\alpha d\mathcal{H}^1(x) \\
&=&
\sum_{k=1}^\infty \int_{M_k} \theta_k(x)^\alpha d\mathcal{H}^1(x) 
=\sum_{k=1}^{\infty} \mathbf{M}_\alpha(T_k).
\end{eqnarray*}
\end{proof}

\begin{lemma}
\label{lemma: finite_n}
For any transport capacity $c>0$ and any $\vec{T}\in Path_c(\mu^-,\mu^+)$, there exists a constant $N(c)\in \mathbb{N}$ with
\[N(c) < \frac{2\|\mu^-\|}{c} +1,\]
and $\vec{T'}=(T'_1,T'_2,\cdots,T'_{N(c)})\in Path_c(\mu^-,\mu^+)$ such that $\mathbf{M}_\alpha(\vec{T'}) \le \mathbf{M}_\alpha(\vec{T})$.
\end{lemma}

\begin{proof}
Since $\sum_{k=1}^{\infty} \| \mu_k^-\|=\sum_{k=1}^{\infty} \| \mu_k^+\| < \infty$, and $\|\mu_k^-\| = \|\mu_k^+\|$ for each $k$, there exists $N$ such that \[\sum_{k=N}^{\infty} \| \mu_k^-\|=\sum_{k=N}^{\infty} \| \mu_k^+\|< \frac{c}{2}.\] 
For any $\vec{T}=(T_1,T_2,\cdots,T_N,\cdots)\in Path_c(\mu^-,\mu^+)$, denote 
\[
T'_N := \sum_{k=N}^\infty T_k \in Path \left(\sum_{k=N}^\infty \mu_k^-, \sum_{k=N}^\infty\mu_k^+\right).
\] 
Then $\vec{T'} = (T_1,T_2,\cdots,T_{N-1}, T'_N) \in Path_c(\mu^-,\mu^+)$, and by Lemma \ref{lem: M alpha cost triagle inequality},
\begin{eqnarray*}
\mathbf{M}_\alpha(\vec{T'})
&=&  
\sum_{k=1}^{N-1} \mathbf{M}_\alpha (T_k)  + \mathbf{M}_\alpha(\vec{T}'_N)=  
\sum_{k=1}^{N-1} \mathbf{M}_\alpha (T_k)  + \mathbf{M}_\alpha\left(\sum_{k=N}^\infty T_k \right) \\
&\le & \sum_{k=1}^\infty \mathbf{M}_\alpha (T_k) =
\mathbf{M}_\alpha (\vec{T}).
\end{eqnarray*}
As a result, without loss of generality, we may assume that $\vec{T}$ has only finitely number of components, i.e. $\vec{T}=(T_1,T_2,\cdots,T_N)$.

We may further assume that there is at most one $k$ with $1\le k\le N$ satisfying $\|\mu_k^-\| \leq c/2$. 
Indeed, assume for some $1\le i<j\le N$, it holds that $\|\mu_i^-\| \le c/2$ and $\|\mu_j^-\| \le c/2 $. 
Let 
$$\vec{T}^* :=\left(T_1,\cdots, T_i +T_j, \cdots ,T_{j-1},T_{j+1},\cdots,T_N\right).$$
Then
$\vec{T}^* \in Path_c(\mu^-, \mu^+),$
since $\|\mu_i^++\mu_j^+\| = \|\mu_i^-+\mu_j^-\| = \|\mu_i^-\| + \|\mu_j^-\| \le c$. 
Also, 
\[\mathbf{M}_\alpha(\vec{T}^*)= 
\sum_{k \neq i,j}\mathbf{M}_\alpha(T_k) \ + \mathbf{M}_\alpha(T_i+T_j) \le 
\sum_{k \neq i,j}\mathbf{M}_\alpha(T_k) \ + \mathbf{M}_\alpha(T_i) + \mathbf{M}_\alpha(T_j)  =   \mathbf{M}_\alpha(\vec{T}).
\] 

Thus, replacing $\vec{T}$ by $\vec{T}^*$ if necessary, we may assume that there is at most one $k$, with $ 1\le k\le N$, satisfying $\|\mu_k^-\| \le c/2$. 
Hence,
\[\|\mu^-\|=\sum_{k=1}^N \|\mu_k^-\| > (N-1)\frac{c}{2},\]
which implies $N<  2\|\mu^-\|/c +1$. 
\end{proof}

The following theorem says that Problem \ref{problem: main} has a solution when $\alpha\in (1-1/m,1]$.
\begin{theorem} \label{thm: existence of optimal transport paths}
Let $\mu^-$ and $\mu^+$ be two Radon measures on $ \mathbb{R}^m$ with equal mass, supported on compact sets. 
For any $\alpha\in (1-1/m,1]$ and $c>0$, there exists a transport multi-path $\vec{T}^*=(T_1,T_2,\cdots, T_N) \in Path_c(\mu^-,\mu^+)$ of finite many components such that 
\[\mathbf{M}_\alpha (\vec{T}^*)\le \mathbf{M}_\alpha (\vec{T})\]
for all $\vec{T} \in Path_c(\mu^-,\mu^+)$. Moreover, the number of components $N$ of $\vec{T}^*$ is less than $ \frac{2}{c}\|\mu^-\| +1 $.
\end{theorem}

\begin{proof}
We first show that there exists a transport path $\vec{S} \in Path_c(\mu^-, \mu^+)$ satisfying (\ref{eqn: problem_1}) with $\mathbf{M}_\alpha(\vec{S})<\infty$. Indeed,
since both $\mu^-$ and $\mu^+$ are supported on a compact set, by the existence theorem in \cite{xia2003} we can find $S\in Path(\mu^-,\mu^+)$ with $\mathbf{M}_\alpha(S)<+\infty$ for $\alpha \in (1-1/m ,1]$. 

Pick $L\in \mathbb{N}$ large enough so that $\|\mu^+\| = \|\mu^-\| \le cL$. Then
$$\vec{S}:=\left(\frac{1}{L}S,\frac{1}{L}S,\cdots, \frac{1}{L}S \right) \in Path_c(\mu^-, \mu^+)$$ with $\mu_k^{\pm}=\frac{1}{L}\mu^{\pm}$ for $k=1,2,\cdots, L$.
Moreover, $$\mathbf{M}_{\alpha}(\vec{S})= \sum_{k=1}^L  \mathbf{M}_{\alpha}(S /L) = \sum_{k=1}^L L^{-\alpha} \mathbf{M}_{\alpha}(S) = L^{1-\alpha} \mathbf{M}_{\alpha}(S)<\infty.$$

Now, let $\{\vec{T}^{(n)}\}$ be any $\mathbf{M}_\alpha$-minimizing sequence in $Path_c(\mu^-, \mu^+)$ with $$\mathbf{M}_{\alpha}(\vec{T}^{(n)})\le \mathbf{M}_{\alpha}(\vec{S}).$$
By Lemma \ref{lemma: finite_n}, without loss of generality, we may assume that 
each $\vec{T}^{(n)}=\left(T_1^n, T_2^n,\cdots, T_N^n\right)$ with $N=N(c)< \frac{2\|\mu^-\|}{c} +1$. 
For each $1\le k \le N$ and $n\in \mathbb{N}$, let $T_k^n = \underline{\underline{\tau}}(M_{k,n},\theta_{k,n},\xi_{k,n})$ with $\theta_{k,n}(x) \le \|\mu_k^-\| \le c$, then
\begin{eqnarray*}
\mathbf{M}(T_k^n) 
&=& 
\int_{M_{k,n}} \theta_{k,n} \,\mathcal{H}^1(x) \\
&=&  
\int_{M_{k,n}} \theta_{k,n}^\alpha \cdot \theta_{k,n}^{1-\alpha} \,\mathcal{H}^1(x)  \le
c^{1-\alpha} \int_{M_{k,n}} \theta_{k,n}^\alpha  \,\mathcal{H}^1(x) = 
c^{1-\alpha} \mathbf{M_\alpha}(T_k^n).   
\end{eqnarray*}
Hence,
\[ 
\mathbf{M}(T_k^n)\le 
c^{1-\alpha} \mathbf{M}_\alpha(T_k^n)\le 
c^{1-\alpha} \mathbf{M}_\alpha(\vec{T}^{(n)})\le 
c^{1-\alpha} \mathbf{M}_{\alpha}(\vec{S})<\infty.
\]

By the compactness of rectifiable currents, each sequence $$\{T_k^n\}_{n=1}^\infty$$ sub-sequentially converges to some rectifiable current $T_k$ for $k=1,2,\cdots, N$. 
Since $N$ is finite, we may assume they have the same convergent subsequence. As a result, we have a convergent subsequence of $\{\vec{T}^{(n)}=\left(T_1^n, T_2^n,\cdots, T_N^n\right)\}$ with limit 
$\vec{T}^*=(T_1, T_2, \cdots, T_ N)$. It is easy to see that $\vec{T}^*\in Path_c(\mu^-, \mu^+)$. Also, by the lower-semicontinuity of $\mathbf{M}_{\alpha}$-mass, this multi-path $\vec{T}^*$ is the desired solution for Problem \ref{problem: main}. 
\end{proof}

\section{Properties of optimal transport multi-paths}
This section aims to study some non-trivial properties of optimal transport multi-paths. To do so, we first introduce some notations.

Given a transport multi-path $(T_1,T_2,\cdots,T_N) \in Path_c(\mu^-,\mu^+)$, it follows that each component $T_k \in Path(\mu_k^-,\mu_k^+)$ with $\partial T_k = \mu_k^+ - \mu_k^-$.
Denote 
$\partial^{-} T_k := \mu_k^-$ as the source measure and denote $ \partial^{+} T_k := \mu_k^+$ as the target measure,
then by definition of transport path we have $\|\partial^-T_k\| = \|\partial^+T_k\|$. 
The conditions for transport multi-paths in $Path_c(\mu^-,\mu^+)$ can be expressed as 
$$\mu^- = \sum_{k=1}^N \partial^-T_k, \ \mu^+ = \sum_{k=1}^N \partial^+T_k, \ \|\partial^-T_k\| = \|\partial^+T_k\|  \le c.$$

We now introduce a technical result that will be used later to determine whether a transport multi-path is optimal or not.
\begin{theorem}
\label{eq:thm_perturbation_capacity}
Given $\mu^- = \sum_{k=1}^N \mu_k^-, \mu^+ = \sum_{k=1}^N \mu_k^+, \|\mu_k^-\|=\|\mu_k^+\| \le c,$ and a transport multi-path $$\vec{T}=(T_1,T_2,\cdots,T_N)\in Path_c(\mu^-,\mu^+).$$
Suppose $\vec{S} = (S_1,S_2,\cdots,S_N)$ is a list of rectifiable 1-currents such that for each $k=1,2,\cdots, N$, its component $S_k$ satisfies the following conditions:
\begin{itemize}
\item [(1)] $S_k$ is on $T_k$ in the sense of Definition \ref{def: cycle-free current}(a);

\item [(2)] In the sense of signed measures, $\partial S_k = \rho_k(x) \partial T_k$ with
$| \rho_k(x) | \le 1$ and 
\begin{equation}
\label{eqn: balance_S_T}
\sum_{k=1}^N \rho_k(x) \partial^-T_k = 0,\ \sum_{k=1}^N \rho_k(x) \partial^+T_k = 0; 
\end{equation}

\item [(3)] $\|\partial^-(T_k \pm S_k)\|\le c$ and $\|\partial^+(T_k \pm S_k)\| \le c$.
\end{itemize}
Then for any $\epsilon \in [-1,1]$, $$\vec{T} + \epsilon\vec{S} := (T_1+\epsilon S_1,T_2+\epsilon S_2,\cdots, T_N+\epsilon S_N)\in Path_c(\mu^-,\mu^+),$$ and for $\alpha \in [0,1]$,
\begin{equation}
\label{eqn: m_alpha_comp}
    \min \left\{
\mathbf{M}_\alpha(\vec{T} + \vec{S}),
\mathbf{M}_\alpha(\vec{T} - \vec{S}) \right\}
\le \mathbf{M}_\alpha(\vec{T}).
\end{equation}
Moreover, if $\vec{T}$ is $\mathbf{M}_{\alpha}$-optimal for some $\alpha \in (0,1)$, 
then $\vec{S}=\vec{0}$ in the sense that $S_k=0$ for all $k=1,2,\cdots, N$.
\end{theorem}

\begin{proof}
For each $k$, since $$\partial (T_k + \epsilon S_k) = \partial T_k + \epsilon\partial S_k = (1+\epsilon \rho_k(x))\partial T_k$$ and $1+\epsilon \rho_k(x)\ge 0$ because $|\epsilon|\le 1$ and $|\rho_k(x)|\le 1$, it follows that 
$$
\partial^-(T_k+\epsilon S_k) = (1+\epsilon \rho_k(x)) \partial^- T_k,\ \mathrm{ and }\ 
\partial^+(T_k+\epsilon S_k) = (1+\epsilon \rho_k(x)) \partial^+ T_k.
$$
By (\ref{eqn: balance_S_T}), this implies that
\begin{eqnarray*}
\sum_{k=1}^N \partial^- (T_k +\epsilon  S_k) &=& 
\sum_{k=1}^N (1+\epsilon \rho_k(x)) \partial^-T_k \\
&=&
\sum_{k=1}^N \partial^-T_k + \epsilon \sum_{k=1}^N \rho_k(x) \partial^-T_k =
\sum_{k=1}^N \partial^-T_k  = 
\mu^-,
\end{eqnarray*}
and similarly
\begin{eqnarray*}
\sum_{k=1}^N \partial^+ (T_k +\epsilon  S_k) &=& 
\sum_{k=1}^N (1+\epsilon \rho_k(x)) \partial^+T_k \\
&=&
\sum_{k=1}^N \partial^+T_k + \epsilon \sum_{k=1}^N \rho_k(x) \partial^+T_k =
\sum_{k=1}^N \partial^+T_k  = 
\mu^+ .
\end{eqnarray*}
Also, since 
$
\| \partial^-(T_k +\epsilon S_k)\| 
=
\int_{X} 1+\epsilon\rho_k(x) \,d(\partial^-T_k)$ 
is linear in $\epsilon \in [-1,1]$, it follows that
$$\|\partial^-(T_k+\epsilon S_k)\|\leq \max\left\{\|\partial^-(T_k+S_k)\|, \|\partial^-(T_k-S_k)\|\right\} 
\le
c.$$
Similarly, $\|\partial^+(T_k+\epsilon S_k)\| \le c$. As a result, $\vec{T}+\epsilon\vec{S} \in  Path_c(\mu^-,\mu^+)$ for each $\epsilon\in [-1,1]$.

To show (\ref{eqn: m_alpha_comp}), without loss of the generality, we may assume that $\mathbf{M}_{\alpha}(\vec{T})<\infty$, i.e., $\mathbf{M}_{\alpha}(T_k)<\infty$ for each $k$.
For each $k=1,2,\cdots,N$, denote $T_k=\underline{\underline{\tau}}(M_k,\theta_k,\xi_k)$ and $S_k=\underline{\underline{\tau}}(N_k,\phi_k,\zeta_k)$. By definition, $S_k$ is on $T_k$ means that $\mathcal{H}^1(N_k\setminus M_k)=0$, and $\phi_k(x)\le \theta_k(x)$ for $\mathcal{H}^1$ almost every $ x\in N_k$. Note that  $\xi_k(x) = \pm \zeta_k (x)$ for
$\mathcal{H}^1$ almost every $ x\in N_k$, since two rectifiable sets have the same tangent almost everywhere on their intersection. 
Thus, it follows that
\[\mathbf{M}_\alpha (T_k+\epsilon S_k)=
\int_{M_k} |\theta_k(x) + \epsilon \phi_k(x) \langle \xi_k(x), \zeta_k(x)\rangle |^\alpha d\mathcal{H}^1(x),\]
where we use the convention that $\phi_k(x)=0, \zeta_k(x)=\xi_k(x)$ when $x\in M_k\setminus N_k$ and $\langle \xi_k(x), \zeta_k(x)\rangle=\pm 1$ for a.e. $x\in N_k\cap M_k$.
Note also that when $\epsilon \in [-1,1]$, since $\phi_k(x)\le \theta_k(x)$, we have $ \theta_k(x) + \epsilon \phi_k(x) \langle \xi_k(x),   \zeta_k(x)\rangle\ge 0$, and hence
\[\mathbf{M}_\alpha (T_k+\epsilon S_k)=
\int_{M_k} \left(\theta_k(x) + \epsilon \phi_k(x) \langle \xi_k(x), \zeta_k(x)\rangle \right)^\alpha d\mathcal{H}^1(x).\]
Since $\mathbf{M}_\alpha (T_k)=\int_{M_k} \theta_k(x)^\alpha d\mathcal{H}^1(x)<\infty$ and $\phi_k(x)\le \theta_k(x)$, one may take derivatives directly and get
\[\frac{d^2}{d\epsilon^2}\mathbf{M}_\alpha (T_k+\epsilon S_k)=
\alpha(\alpha-1) 
\int_{M_k}\left(\theta_k(x) + \epsilon \phi_k(x) \langle \xi_k(x), \zeta_k(x)\rangle \right)^{\alpha-2} \phi_k(x)^2\le 0.
\]
This implies $\mathbf{M}_\alpha (\vec{T}+ \epsilon \vec{S})=\sum_{k=1}^N \mathbf{M}_\alpha (T_k+\epsilon S_k) $ is a concave function of $\epsilon$ and hence it reaches minimum value at one of its endpoint $\epsilon=\pm 1$. 
As a result, $\min  \{ \mathbf{M}_\alpha(\vec{T}+\vec{S}), \mathbf{M}_\alpha(\vec{T}-\vec{S})  \} 
\le \mathbf{M}_\alpha(\vec{T})$ as desired.

Now assume that $\vec{T}$ is $\mathbf{M}_{\alpha}$-optimal for some $\alpha\in (0,1)$ but $\vec{S} = (S_1,S_2,\cdots,S_N)$ is non-zero. i.e. there exists $k \in \{1,2,\cdots,N\}$ such that $S_{k}$ is a non-vanishing current.  
Then,
\[\left.\frac{d^2}{d\epsilon^2} \right\vert_{\epsilon=0}\mathbf{M}_\alpha (T_k+\epsilon S_k)=
\alpha(\alpha-1) 
\int_{M_k} \theta_k(x)^{\alpha-2} \phi_k(x)^2<0
\]
due to the facts 
\[\int_{M_k} \theta_k(x)^{\alpha-2} \phi_k(x)^2= \int_{N_k} \theta_k(x)^{\alpha-2} \phi_k(x)^2\ge \int_{N_k} \phi_k(x)^{\alpha-2} \phi_k(x)^2=\mathbf{M}_{\alpha}(S_k)>0,\]
and $\alpha\in (0,1)$. This implies that $\mathbf{M}_\alpha(\vec{T}+ \epsilon \vec{S})$  cannot achieve a local minimum at $\epsilon=0$, contradicting with $\vec{T}$ is $\mathbf{M}_\alpha$-optimal.
\end{proof}

\begin{remark} 
Note that an obvious necessary condition for $\vec{S}=(S_1,S_2,\cdots, S_N)$ is given by condition (2):
$$\partial \left( \sum_{k=1}^N S_k \right)=
\sum_{k=1}^N \partial S_k  =
\sum_{k=1}^N \rho_k(x) \partial T_k =
\sum_{k=1}^N \rho_k(x) \partial^+ T_k -\sum_{k=1}^N \rho_k(x) \partial T_k^- =0.$$
In general, a component $\partial S_k$ of $\vec{S}$ does not necessarily vanish. Nevertheless, when each $\partial S_k =0$, conditions (2) and (3) are automatically satisfied, since we have $\rho_k(x)=0$ and $\partial(T_k \pm S_k) = \partial T_k$.
\end{remark}

We now use Theorem \ref{eq:thm_perturbation_capacity} to study properties of optimal transport multi-paths between atomic measures. In the following content, let
\begin{equation}\label{eq: capacity atomic measures}  
\mu^- =\sum_{i=1}^{N_1} m'_i \delta_{x_i} \text{ and } 
\mu^+ = \sum_{j=1}^{N_2} m_j \delta_{y_j} \text{ with } \sum_{i=1}^{N_1} m'_i=\sum_{j=1}^{N_2} m_j.
\end{equation}
By Theorem \ref{thm: existence of optimal transport paths}, there exists an $\mathbf{M}_{\alpha}$-minimizer $\vec{T}^*=(T_1,T_2,\cdots, T_N)$ among all transport multi-paths $\vec{T}\in Path_c(\mu^-, \mu^+)$. Here, each component $T_k$ is an optimal transport path from $\mu_k^-$ to $\mu^+_k$, and $\sum_{k=1}^N\mu_k^{\pm}=\mu^{\pm}$. 

Besides knowing $T_k$ satisfies some basic properties of an optimal transport path, we are interested in finding some non-trivial properties of $T_k$. Our first result in this direction is about the map-compatibility.

\begin{theorem}
\label{thm: X_j's mutually disjoint}
Let $\mu^-$ and $\mu^+$ be defined as in equation (\ref{eq: capacity atomic measures}), and
$\alpha \in (0,1)$. 
Let $\vec{T}^*= (T_1,T_2,\cdots,T_N) $ be an $\mathbf{M}_{\alpha}$-minimizer in $Path_c(\mu^-, \mu^+)$. Then except for at most $N_2-1$ many $k$'s,
each component $T_k$ of $\vec{T}^*$ is map-compatible.
\end{theorem}

\begin{proof}
For each $k=1,2,\cdots, N$, since $T_k$ is an optimal transport path, it is an acyclic normal 1-current. Thus, it has a good decomposition $\eta_k$.  Using similar notation as in equation (\ref{eqn: X_j}), we denote 
\begin{equation}
\label{eqn: X_j-defn}
    X_j(\eta_k) := \{x_i\in X : \eta_k(\Gamma_{x_i,y_j})>0\}.
\end{equation}
To prove the theorem, by Proposition \ref{prop: X_1 intersect X_2 = empty gives transport map}, it is sufficient to prove the collection of sets
$$\{X_j (\eta_k) : j=1,2,\cdots, N_2 \} $$
are mutually disjoint, except for at most $N_2-1$ many $k$'s. For the sake of contradiction, we assume there are $N_2$ collections of sets 
$$\left\{ X_j(\eta_{k_\ell}) \,:\, j=1,2,\cdots, N_2 \right\},\text{ for } \ell=1,2,\cdots,N_2, $$
that are not mutually disjoint.  
Thus for each $\ell$, there exist $x_{i_\ell},y_{j_\ell},y_{j'_\ell},$ such that 
$x_{i_\ell} \in X_{j_\ell}(\eta_{k_\ell}) \cap X_{j'_\ell}(\eta_{k_\ell}).$  That is,
\begin{equation}
\label{eqn: eta_positive}
    \eta_{k_\ell}(\Gamma_{x_{i_\ell},y_{j_\ell}}) >0 \text{ and }\eta_{k_\ell}(\Gamma_{x_{i_\ell},y_{j'_\ell}}) >0.
\end{equation}
In the following, to get a contradiction with the optimality of $\vec{T}^*$, we will build a non-vanishing $\vec{S}$ satisfying the conditions in Theorem \ref{eq:thm_perturbation_capacity}.

Let $M$ be the $N_2\times N_2$ matrix
$$M:=\begin{bmatrix}
e_{j_1'} - e_{j_1}\\
e_{j_2'} - e_{j_2}\\
\vdots \\
e_{j_{N_2}'} - e_{j_{N_2}}\\
\end{bmatrix},$$
where for each $k=1,2\cdots, N_2$,  $e_{k}$ denotes the row vector $ (0,\cdots,1,\cdots,0)$ in $\mathbb{R}^{N_2}$ with the nonzero number $1$ at its $k$-th position. Since 
$$
M
\begin{bmatrix}
1\\
1\\
\vdots \\
1\\
\end{bmatrix}
=
\begin{bmatrix}
0\\
0\\
\vdots \\
0\\
\end{bmatrix}, $$
it follows that $rank\left(
M\right) <N_2$. 
Thus, there exists a nonzero row vector $$[c_1,c_2,\cdots,c_{N_2}] \in\mathbb{R}^{N_2}$$ such that
\begin{equation}
\label{eq: define c_l small}
\begin{bmatrix}
c_1,c_2,\cdots,c_{N_2}
\end{bmatrix}
M
=
\begin{bmatrix}
0,0,\cdots,0
\end{bmatrix}.
\end{equation}
By (\ref {eqn: eta_positive}) and (\ref{eq: define c_l small}), we may further assume that 
\begin{equation}
\label{eqn: c_bounds}
    0 < \max\{ |c_\ell| : 1\le \ell\le N_2\} \le \min\{ \eta_{k_\ell}(\Gamma_{x_{i_\ell},y_{j_\ell}}), \ \eta_{k_\ell}(\Gamma_{x_{i_\ell},y_{j'_\ell}}) : 1\le \ell\le N_2\}.
\end{equation}

Since $\eta_k$ is a good decomposition of $T_k\in Path(\mu_k^-,\mu_k^+)$, it follows that
\[T_k = \int_{\Gamma} I_\gamma  d\eta_k=\sum_{i=1}^{N_1} \sum_{j=1}^{N_2} \int_{\Gamma_{x_i,y_j}} I_\gamma  d\eta_k.\]

Let 
$$S_{k_\ell}:=  -\frac{c_{\ell}}{\eta_{k_\ell} ( \Gamma_{x_{i_\ell},y_{j_\ell} }) } 
\int_{\Gamma_{x_{i_\ell},y_{j_\ell} } }  I_\gamma d \eta_{k_\ell} + 
\frac{c_{\ell}}{\eta_{k_\ell} ( \Gamma_{x_{i_\ell},y_{j'_\ell} }) } 
\int_{\Gamma_{x_{i_\ell},y_{j'_\ell} } }  I_\gamma d \eta_{k_\ell},$$
for $\ell= 1,2,\cdots,N_2$, and $S_k := 0$ for any other $k$'s. 
We now check that $\vec{S}=(S_1,S_2,\cdots, S_N)$ satisfies the conditions of Theorem \ref{eq:thm_perturbation_capacity}. By (\ref{eqn: c_bounds}), each $S_k$ is on $T_k$.
Since 
\begin{equation}
\label{eqn boundary_S_k}
    \partial {S}_{k_\ell} =c_\ell \delta_{y_{j'_\ell}} - c_\ell \delta_{y_{j_\ell}}, \text{ for } \ell=1,2,\cdots, N_2, \text{ and } \partial S_k = 0 \text{ for any other } k,
\end{equation}
 it follows that $\partial S_k = \rho_k(x) \partial T_k$, where $\rho_k(x)$ is defined as 
\begin{equation}
\label{eqn: rho_k}
\rho_{k_\ell}(x) = \left\{ 
\begin{array}{cl}
-c_\ell/\sum_{i} \eta_{k_\ell}(\Gamma_{x_i,y_{j_\ell}})    & \text{ if } x=y_{j_\ell} \\
c_\ell/\sum_{i}\eta_{k_\ell}(\Gamma_{x_i,y_{j'_\ell}})     & \text{ if } x=y_{j_\ell'} \\
0  & \text{ otherwise,}
\end{array}\right.
\end{equation}
for $\ell=1,2,\cdots,N_2$, and $\rho_k(x) =0$ for any other $k$'s. Therefore, we have $|\rho_k(x)| \le 1$, for all $k$ and all $x$.
Moreover, since $\partial^- T_k (\{x\})=0$ for $x \not\in \{x_1,x_2,\cdots,x_{N_1}\}$ and $\rho_k(x)=0$ for $x\in \{x_1,x_2,\cdots,x_{N_1}\}$, by (\ref{eqn: rho_k}),  we have
$$\sum_{k=1}^N \rho_k(x) \partial^- T_k  =\sum_{k=1}^N 0\cdot \partial^- T_k = 0. $$
Using it, (\ref{eqn boundary_S_k}) and (\ref{eq: define c_l small}), 
\begin{eqnarray*}
\sum_{k=1}^N \rho_k(x) \partial^+T_k &=& 
\sum_{k=1}^N \rho_k(x) \partial^+T_k
-\sum_{k=1}^N \rho_k(x) \partial^-T_k \\
&=& 
\sum_{k=1}^N \partial S_k = 
\sum_{\ell=1}^{N_2} \partial S_{k_\ell}
=
\sum_{\ell=1}^{N_2} c_\ell \cdot (\delta_{y_{j'_\ell}} - \delta_{y_{j_\ell}} ) \\
&=&
\begin{bmatrix}
c_1,c_2,\cdots,c_{N_2}
\end{bmatrix}
M
\begin{bmatrix}
\delta_{y_1}\\
\delta_{y_2}\\
\vdots   \\
\delta_{y_{N_2}}
\end{bmatrix}=
\begin{bmatrix}
0,0,\cdots,0
\end{bmatrix}
\begin{bmatrix}
\delta_{y_1}\\
\delta_{y_2}\\
\vdots   \\
\delta_{y_{N_2}}
\end{bmatrix}  = 0.
\end{eqnarray*}

Furthermore, by (\ref{eqn boundary_S_k}), for $k=k_\ell$, with $\ell=1,2,\cdots,N_2$,
$$\|\partial^-(T_{k_\ell} \pm S_{k_\ell})\| =  \|\partial^-T_{k_\ell} \| \le c ,\quad 
\|\partial^+(T_{k_\ell} \pm S_{k_\ell})\| = \mp c_\ell \pm c_\ell +\|\partial^+ T_{k_\ell}\| \le c ,$$
and for $k \not= k_\ell$,
$\|\partial^-(T_k \pm S_k)\| = \|\partial^-T_k\| \le c$ and 
$\|\partial^+(T_k \pm S_k)\| = \|\partial^+T_k\| \le c,$ hold trivially.

As a result, $\vec{S}=(S_1,S_2,\cdots, S_N)$ satisfies the conditions of Theorem \ref{eq:thm_perturbation_capacity}.
Since $\vec{T}^*$ is optimal, by Theorem \ref{eq:thm_perturbation_capacity}, we have $\vec{S} = \vec{0}$.  
On the other hand, since the vector $[c_1,c_2,\cdots,c_{N_2}]$ is non-zero, at least one $S_{k_\ell}$ is non-zero. This leads to a contradiction.
\end{proof}

We now further explore other non-trivial properties of $\vec{T}^*$. Note that Theorem \ref{thm: X_j's mutually disjoint} says that with at most $N_2-1$ exceptions, $|X_{j_1} (\eta_k) \cap X_{j_2} (\eta_k)| =0$, i.e. $T_k$ has a ``zero common-source property". 
Since each $T_k$ is an optimal transport path from $\mu_k^-$ to $\mu_k^+$, by Proposition \ref{prop: single component less than 1 point}, $T_k$ has the ``simple common-source property" in the sense that there exists a good decomposition $\eta_k$ of $T_k$ such that $|X_{j_1} (\eta_k) \cap X_{j_2} (\eta_k)| \le 1$ for any $1 \le j_1 < j_2 \le N_2$. 
We now generalize this simple common-source property to its counterparts between different components.

\begin{theorem}[Simple Common-Source Property]
\label{General source overlap property}
Let $\mu^-$ and $\mu^+$ be defined as in equation (\ref{eq: capacity atomic measures}), and
$\alpha \in (0,1)$.
    Let $\vec{T}^*= (T_1,T_2,\cdots,T_N)$ be an $\mathbf{M}_{\alpha}$-minimizer in $Path_c(\mu^-, \mu^+)$.
For each $k=1,2,\cdots, N$, let $\eta_k$ be any good decomposition of $T_k$.
Then for all pairs $1 \le j_1 \leq j_2 \le N_2$, $1 \le k_1 < k_2 \le N$, it holds that
\begin{equation}
\label{eqn: X_j_inequality}
    |X_{j_1}(\eta_{k_1}) \cap X_{j_2}(\eta_{k_2}) | \le 1 .
\end{equation}
\end{theorem}

\begin{proof}
For the sake of contradiction, assume 
$$|X_{j_1}(\eta_{k_1}) \cap X_{j_2}(\eta_{k_2}) | \ge 2$$
for some $1 \le j_1 \leq j_2 \le N_2$ and $1 \le k_1 <k_2 \le N$. Thus, there exists two points, say $x_1,x_2 \in X_{j_1}(\eta_{k_1}) \cap X_{j_2}(\eta_{k_2}) $, with $x_1 \not=x_2$. By definition of $X_j(\eta_k)$ given in (\ref{eqn: X_j-defn}), it means that
$$\eta_{k_1} (\Gamma_{x_1,y_{j_1}})>0,\  \eta_{k_1} (\Gamma_{x_2,y_{j_1}})>0, \ \eta_{k_2} (\Gamma_{x_1,y_{j_2}})>0, \text{ and }\eta_{k_2} (\Gamma_{x_2,y_{j_2}})>0.$$

As in the proof of Theorem \ref{thm: X_j's mutually disjoint}, to get a contradiction with the optimality of $\vec{T}^*$, we will build a non-vanishing $\vec{S}$ satisfying the conditions in Theorem \ref{eq:thm_perturbation_capacity}. For each $k$, since $\eta_k$ is a good decomposition of $T_k\in Path(\mu_k^-,\mu_k^+)$, it follows that
\[T_k = \int_{\Gamma} I_\gamma  d\eta_k=\sum_{i=1}^{N_1} \sum_{j=1}^{N_2} \int_{\Gamma_{x_i,y_j}} I_\gamma  d\eta_k,\]
and thus
$$\partial T_{k} = \sum_{i=1}^{N_1} \sum_{j=1}^{N_2}  \eta_{k}(\Gamma_{x_i,y_j}) (\delta_{y_j}-\delta_{x_i}).$$

Let
$$\epsilon_0:= \frac{1}{4} \min\{\eta_{k_1} (\Gamma_{x_1,y_{j_1}}), \eta_{k_1} (\Gamma_{x_2,y_{j_1}}), \eta_{k_2} (\Gamma_{x_1,y_{j_2}}), \eta_{k_2} (\Gamma_{x_2,y_{j_2}})\}>0,$$
and define 
\begin{eqnarray*}
&&S_{k_1}:=\ \  \frac{\epsilon_0}{\eta_{k_1} (\Gamma_{x_1,y_{j_1}})} \int_{\Gamma_{x_1,y_{j_1}}} I_\gamma d\eta_{k_1} - \frac{\epsilon_0}{\eta_{k_1} (\Gamma_{x_2,y_{j_1}})} \int_{\Gamma_{x_2,y_{j_1}}} I_\gamma d\eta_{k_1},\\
&&S_{k_2}:= -\frac{\epsilon_0}{\eta_{k_2} (\Gamma_{x_1,y_{j_2}})} \int_{\Gamma_{x_1,y_{j_2}}} I_\gamma d\eta_{k_2} + \frac{\epsilon_0}{\eta_{k_2} (\Gamma_{x_2,y_{j_2}})} \int_{\Gamma_{x_2,y_{j_2}}} I_\gamma d\eta_{k_2},\\
&&S_k \ :=0 \text{ for } k \not= k_1,k_2.
\end{eqnarray*}
Note that
\begin{eqnarray*}
\partial S_{k_1}
&=& \frac{\epsilon_0}{\eta_{k_1} (\Gamma_{x_1,y_{j_1}})} \int_{\Gamma_{x_1,y_{j_1}}} (\delta_{y_{j_1}}-\delta_{x_1}) d\eta_{k_1} - \frac{\epsilon_0}{\eta_{k_1} (\Gamma_{x_2,y_{j_1}})} \int_{\Gamma_{x_2,y_{j_1}}} (\delta_{y_{j_1}}-\delta_{x_2}) d\eta_{k_1} \\
&=& \epsilon_0 \delta_{x_2} - \epsilon_0 \delta_{x_1},
\end{eqnarray*}
and similarly $\partial S_{k_2}=\epsilon_0 \delta_{x_1} - \epsilon_0 \delta_{x_2}$. Since $x_1 \not= x_2$, both $S_{k_1}$ and $S_{k_2}$ are non-vanishing currents on $T_{k_1}$ and $T_{k_2}$, respectively.
Using it, we may express $\partial S_{k_1} = \rho_{k_1}(x) \partial T_{k_1}$, where 
$$\rho_{k_1}(x_1)= \frac{\epsilon_0}{ \sum_{j=1}^{N_2}\eta_{k_1} (\Gamma_{x_1,y_j}) }, \quad
\rho_{k_1}(x_2)= -\frac{\epsilon_0}{ \sum_{j=1}^{N_2}\eta_{k_1} (\Gamma_{x_2,y_j}) }, 
$$
and $\rho_{k_1}(x)=0$ for any $x\in spt(\partial T_{k_1})\setminus\{x_1, x_2\}=\{x_3,x_4,\cdots, x_{N_1}, y_1,y_2,\cdots, y_{N_2}\}$.
Similarly, we have $\partial S_{k_2} = \rho_{k_2}(x) \partial T_{k_2}$, where 
$$\rho_{k_2}(x_1)= -\frac{\epsilon_0}{ \sum_{j=1}^{N_2}\eta_{k_2} (\Gamma_{x_1,y_j}) }, \quad
\rho_{k_2}(x_2)= \frac{\epsilon_0}{ \sum_{j=1}^{N_2}\eta_{k_2} (\Gamma_{x_2,y_j}) }, 
$$
and $\rho_{k_2}(x)=0$ for any $x\in spt(\partial T_{k_2})\setminus\{x_1, x_2\}$.
For $k \not= k_1,k_2$, since $S_k=0$, we have $\rho_k(x) = 0$.
As a result, we express $\partial S_k = \rho_k(x) \partial T_k$ with $|\rho_k(x)| \le 1$ for each $k=1,2,\cdots,N$.

Moreover, by the values of $\rho_k(x)$ given above,
\begin{eqnarray*}
&& \sum_{k=1}^N \rho_k(x)\partial^-T_k 
=\rho_{k_1}(x)\partial^-T_{k_1}+\rho_{k_2}(x)\partial^-T_{k_2} \\
&=& 
\frac{\epsilon_0}{\sum_{j=1}^{N_2}\eta_{k_1}(\Gamma_{x_1,y_j})}\cdot \sum_{j=1}^{N_2}\eta_{k_1}(\Gamma_{x_1,y_j})\delta_{x_1} -
\frac{\epsilon_0}{\sum_{j=1}^{N_2}\eta_{k_1}(\Gamma_{x_2,y_j})}\cdot \sum_{j=1}^{N_2}\eta_{k_1}(\Gamma_{x_2,y_j})\delta_{x_2}   \\
&& -
\frac{\epsilon_0}{\sum_{j=1}^{N_2}\eta_{k_2}(\Gamma_{x_1,y_j})}\cdot \sum_{j=1}^{N_2}\eta_{k_2}(\Gamma_{x_1,y_j})\delta_{x_1} + 
\frac{\epsilon_0}{\sum_{j=1}^{N_2}\eta_{k_2}(\Gamma_{x_2,y_j})}\cdot \sum_{j=1}^{N_2}\eta_{k_2}(\Gamma_{x_2,y_j})\delta_{x_2}\\
&=&
0,    
\end{eqnarray*}
and
$$\sum_{k=1}^N \rho_k(x)\partial^+T_k = \sum_{k=1}^N 0 = 0.$$

We now check condition (3) in the statement of Theorem \ref{eq:thm_perturbation_capacity}.
\begin{itemize}
    \item When $k={k_1}$, $\|\partial^+(T_{k_1} \pm S_{k_1})\| = \|\partial^+T_{k_1}\| \le c ,$ and 
\begin{eqnarray*}
&&\|\partial^- (T_{k_1} \pm S_{k_1} )\| \\
&=& 
\left(\sum_{j=1}^{N_2}\eta_{k_1} (\Gamma_{x_1,y_j}) \right)\pm \epsilon_0 + 
\left(\sum_{j=1}^{N_2}\eta_{k_1} (\Gamma_{x_2,y_j}) \right)\mp \epsilon_0 + \sum_{i=3}^{N_1} \sum_{j=1}^{N_2}\eta_{k_1} (\Gamma_{x_i,y_j})  \\
&=&
\sum_{i=1}^{N_1} \sum_{j=1}^{N_2}\eta_{k_1} (\Gamma_{x_i,y_j}) 
= \|\partial^- T_{k_1}\| \le c.
\end{eqnarray*}
Similarly, for $k={k_2}$ we have $\|\partial^+ (T_{k_2} \pm S_{k_2})\| \le c$ and $\|\partial^- (T_{k_2} \pm S_{k_2})\| \le c$.

\item When $k \not= k_1,k_2$, since $S_k$'s are zero currents, we trivially have
$$\|\partial^-(T_k \pm S_k)\| = \|\partial^-T_k\| \le c, \ 
\|\partial^+(T_k \pm S_k)\| = \|\partial^+T_k\| \le c.$$
\end{itemize}

Thus, $\vec{S}=(S_1,S_2,\cdots, S_N)$ satisfies the conditions of Theorem \ref{eq:thm_perturbation_capacity}.
Since $T$ is optimal, by Theorem \ref{eq:thm_perturbation_capacity}, we have $\vec{S} = \vec{0}$, contradicting with the non-vanishing $S_{k_1},S_{k_2}$ constructed above.
As a result, for all $1 \le j_1 \leq j_2 \le N_2$, $1 \le k_1 <k_2 \le N$,
$$|X_{j_1}(\eta_{k_1}) \cap X_{j_2}(\eta_{k_2}) | \le 1 .$$
\end{proof}

\section{Case study: Single target}
As an example, in this section, we would like to investigate the case of a single target. Namely, let
$$\mu^- = \sum_{i=1}^{N_1} m'_i \delta_{x_i}, \ \mu^+ = m_1 \delta_{y} \text{ with } m_1=\sum_{i=1}^{N_1} m'_i,$$ 
$\vec{T}^*= (T_1,T_2,\cdots,T_N) $ be an $\alpha$-optimal multi-path in $Path_c(\mu^-,\mu^+)$ for some $\alpha \in (0,1)$ and $c>0$. For each $k=1,2,\cdots, N$, let $\eta_k$ be a good decomposition
of $T_k$. Thus, 
\[T_k = \int_{\Gamma} I_\gamma  d\eta_k=\sum_{i=1}^{N_1} \int_{\Gamma_{x_i,y}} I_\gamma  d\eta_k,\]
and 
\begin{equation}
\label{eqn: boundary_T_k}
    \partial T_{k} = \sum_{i=1}^{N_1} \eta_{k}(\Gamma_{x_i,y}) (\delta_{y}-\delta_{x_i})=\left(\sum_{i=1}^{N_1} \eta_{k}(\Gamma_{x_i,y})\right) \delta_{y}-\sum_{i=1}^{N_1} \eta_{k}(\Gamma_{x_i,y}) \delta_{x_i}.
\end{equation}

For simplicity, we also denote 
$$X(\eta_k) := \{x_i\in X : \eta_k(\Gamma_{x_i,y})>0\}.$$
By Theorem \ref{General source overlap property}, when $k_1 \not= k_2$, it holds that
\begin{equation}
\label{eq:supple1}
|X(\eta_{k_1}) \cap X(\eta_{k_2})|\le 1.    
\end{equation}

\begin{proposition}\label{eq:prop_suppeq1}
In the single target case, suppose 
$|X(\eta_{k_1}) \cap X(\eta_{k_2})| = 1$ for some $k_1 \not= k_2$,
then either $\|\mu_{k_1}^-\|=c$ or $\|\mu_{k_2}^-\|=c$.
\end{proposition}

\begin{proof}
Without loss of generality, assume $k_1=1,k_2=2$, and let $X(\eta_{1}) \cap X(\eta_{2})=\{x_1\}$, which implies that
$$
\eta_1(\Gamma_{x_1,y})>0, \text{ and } \eta_2(\Gamma_{x_1,y})>0 .$$

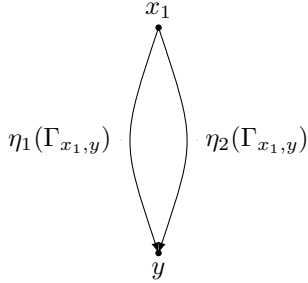
\begin{figure}[h]
\begin{tikzpicture}[>=latex]
\filldraw[black] (0,0) circle (1pt) node[anchor=north]{$y$};
\filldraw[black] (0,3) circle (1pt) node[anchor=south]{$x_1$};

\filldraw[black] (-0.5,1.5) circle (0pt) node[anchor=east]{$\eta_1(\Gamma_{x_1,y})$};
\filldraw[black] ( 0.5,1.5) circle (0pt) node[anchor=west]{$\eta_2(\Gamma_{x_1,y})$};

\draw[->] (0,3) .. controls (-0.5,1.5).. (0,0);
\draw[->] (0,3) .. controls ( 0.5,1.5).. (0,0);

\end{tikzpicture}
\label{fig: supp_eq_1}
\caption{$T_1$ and $T_2$}
\end{figure}

For the sake of contradiction, assume that
both $\|\mu_1^-\|<c$ and $\|\mu_2^-\| < c$. Let 
$$ \epsilon_0 := \min\left\{\eta_1(\Gamma_{x_1,y}),\eta_2(\Gamma_{x_1,y}), c-\|\mu_1^-\|, c-\|\mu_2^-\|\right\}>0.$$
Consider $$\vec{S} = (S_1,S_2,S_3,\cdots,S_N),$$ where 
$$S_1 := \frac{\epsilon_0}{\eta_1(\Gamma_{x_1,y})}\int_{\Gamma_{x_1,y}} I_\gamma d \eta_1, \ S_2:= -\frac{\epsilon_0}{\eta_2(\Gamma_{x_1,y})}\int_{\Gamma_{x_1,y}} I_\gamma d \eta_2,$$
and $S_k :=0$ for $k \ge 3$. By construction, each $S_k$ is on $T_k$ for all $k=1,2,\cdots,N$.
Now we may express $\partial S_{k} = \rho_{k}(x) \partial T_{k}$ as follows.
Since $\partial T_1$ is given in (\ref{eqn: boundary_T_k}) and
$$\partial S_1= \frac{\epsilon_0}{\eta_1(\Gamma_{x_1,y})} \int_{\Gamma_{x_1,y}} \partial I_\gamma d \eta_1 =
\frac{\epsilon_0}{\eta_1(\Gamma_{x_1,y})} \int_{\Gamma_{x_1,y}} (\delta_{y}-\delta_{x_1}) d \eta_1 = \epsilon_0\delta_y -\epsilon_0\delta_{x_1}, $$
it follows that $\rho_1(x_1)= \epsilon_0 / \eta_1(\Gamma_{x_1,y})$, $\rho_1(y)= \epsilon_0 /  \sum_{i=1}^{N_1} \eta_1 (\Gamma_{x_i,y})$, and $\rho_1(x)=0$ for $x\neq x_1,y$.
Note that $|\rho_1(x)| \le 1$. 

In general, for each $k=1,2,\cdots,N$, we get $\partial S_k = \rho_k(x)\partial T_k$, with $|\rho_k(x)| \le 1$ where
$$
\begin{array}{c|ccc}
     & x=x_1 & x=y &  \textrm{otherwise} \\
     \hline
\rho_1(x) & \epsilon_0/\eta_1(\Gamma_{x_1,y}) & \epsilon_0/\sum_{i=1}^{N_1} \eta_1 (\Gamma_{x_i,y}) & 0 \\
\rho_2(x) & -\epsilon_0/\eta_2(\Gamma_{x_1,y}) &-\epsilon_0/\sum_{i=1}^{N_1} \eta_2 (\Gamma_{x_i,y}) & 0 \\ \rho_{k}(x),k\ge 3 & 0 & 0 & 0 \\
\end{array}$$

Moreover, using the values of $\rho_k$ given above, we have
$$\sum_{k=1}^N \rho_k(x)\partial^-T_k = \frac{\epsilon_0}{\eta_1(\Gamma_{x_1,y})}\cdot \eta_1(\Gamma_{x_1,y})\delta_{x_1} -\frac{\epsilon_0}{\eta_2(\Gamma_{x_1,y})}\cdot \eta_2(\Gamma_{x_1,y})\delta_{x_1} = 0  ,$$
and
$$\sum_{k=1}^N \rho_k(x)\partial^+T_k  = \frac{\epsilon_0}{\sum_{i=1}^{N_1} \eta_1 (\Gamma_{x_i,y})} \cdot \sum_{i=1}^{N_1} \eta_1 (\Gamma_{x_i,y}) \delta_y  - 
\frac{\epsilon_0}{\sum_{i=1}^{N_1} \eta_2 (\Gamma_{x_i,y})} \cdot \sum_{i=1}^{N_1} \eta_2 (\Gamma_{x_i,y}) \delta_y = 0 .$$

In the end, we check the condition (3) of Theorem \ref{eq:thm_perturbation_capacity}
\begin{itemize}
\item When $k=1,2$, 
since $\|\mu_k^-\| = \| \mu_k^+ \|$, 
$$\|\partial^-(T_1 \pm S_1)\| = \eta_1 (\Gamma_{x_1,y}) \pm \epsilon_0 +  \sum_{i=2}^{N_1} \eta_1 (\Gamma_{x_i,y})=
\pm \epsilon_0 + \|\mu_1^-\| \le c,$$
$$\|\partial^+(T_1 \pm S_1)\| = \pm \epsilon_0 +  \sum_{i=1}^{N_1} \eta_1 (\Gamma_{x_i,y}) =\pm\epsilon_0 + \|\mu_1^-\| \le c,$$
$$\|\partial^-(T_2 \pm S_2)\| = \eta_2 (\Gamma_{x_1,y}) \mp \epsilon_0 +  \sum_{i=2}^{N_1} \eta_2 (\Gamma_{x_i,y})=
\mp \epsilon_0 + \|\mu_2^-\| \le c,$$
$$\|\partial^+(T_2 \pm S_2)\| = \mp \epsilon_0 +  \sum_{i=1}^{N_1} \eta_2 (\Gamma_{x_i,y}) =\mp\epsilon_0 + \|\mu_2^-\| \le c,$$

\item When $k \ge 3$, 
$$\|\partial^-(T_k \pm S_k)\| = \|\partial^-T_k\| \le c, \  \|\partial^+(T_k \pm S_k)\| = \|\partial^+T_k\| \le c .$$
\end{itemize}
Theorem $\ref{eq:thm_perturbation_capacity}$ implies for $\alpha \in (0,1)$, each $S_k$ is a vanishing current, but $S_1,S_2$ constructed above are non-vanishing, and this leads to a contradiction.
Hence, either $\|\mu_1^-\|=c$ or $\|\mu_2^-\|=c$.
\end{proof}

\begin{corollary} 
\label{cor: intersection capacity}
In the single target case, suppose 
$$\bigcap_{\ell=1}^{n} X(\eta_{k_\ell}) \neq \emptyset
\text{ for some } n\ge 2.$$ 
Then at most one of the $\mu_{k_\ell}$ has $\| \mu_{k_\ell} \| <c$, and any other $\mu_{k_\ell}$'s have mass $\|\mu_{k_\ell}\| =c$.
\end{corollary}

\begin{proof}
Suppose there exist two components $\mu_{k_1},\mu_{k_2}$ with $\|\mu_{k_1}\|<c$ and $\|\mu_{k_2}\| <c$.
By equation (\ref{eq:supple1}), we have $| X(\eta_{k_1}) \cap X(\eta_{k_2})| \le 1$. 
Since $X(\eta_{k_1}) \cap X(\eta_{k_2})$ is non-empty, then $|X(\eta_{k_1}) \cap X(\eta_{k_2})| = 1$.
Proposition \ref{eq:prop_suppeq1} implies  $\|\mu_{k_1}\|=c$ or $\|\mu_{k_2}\|=c$, which leads to a contradiction. 
\end{proof}

In the following examples, denote the line segment from $x$ to $y$ as $\overline{xy}$. 
Also, denote $a \llbracket \gamma \rrbracket$ as the rectifiable $1$-current, with density equals $a$, supported on the curve $\gamma$, and has direction along this curve. Let $\lceil x \rceil$ denote the smallest integer $n$ such that $n\ge x$.

\begin{example}
Suppose $\mu^- = m_1 \delta_{x} ,\ \mu^+ = m_1 \delta_y $ for some $m_1>0$, $c>0$ and $\vec{T}^*=(T_1,T_2,\cdots, T_N) \in Path_c(\mu^-,\mu^+)$ is $\alpha$-optimal for some $0<\alpha<1$. 
Then, up to a permutation of component indices,
$$T_1,T_2,\cdots,T_{N-1} = c\llbracket \overline{xy} \rrbracket,\ T_N = r \llbracket \overline{xy} \rrbracket, $$ with $N=\lceil m_1/c \rceil$ and $r= m_1 -(N-1)c$.

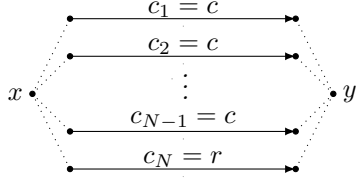
\begin{figure}[h]
\begin{tikzpicture}[>=latex]
\filldraw[black] (-2,0) circle (1pt) node[anchor=east]{$x$};
\filldraw[black] (2,0) circle (1pt) node[anchor=west]{$y$};

\filldraw[black] (-1.5,1) circle (1pt) node[anchor=east]{ };
\filldraw[black] (-1.5,0.5) circle (1pt) node[anchor=east]{ };

\filldraw[black] (-1.5,-0.5) circle (1pt) node[anchor=east]{ };
\filldraw[black] (-1.5,-1) circle (1pt) node[anchor=east]{ };

\filldraw[black] (1.5,1) circle (1pt) node[anchor=east]{ };
\filldraw[black] (1.5,0.5) circle (1pt) node[anchor=east]{ };
\filldraw[black] (-0.15,0.2) circle (0pt) node[anchor=west]{\vdots };
\filldraw[black] (1.5,-0.5) circle (1pt) node[anchor=east]{ };
\filldraw[black] (1.5,-1) circle (1pt) node[anchor=east]{ };

\filldraw[black] (0,0.9) circle (0pt) node[anchor=south]{$c_1=c$ };
\filldraw[black] (0,0.4) circle (0pt) node[anchor=south]{$c_2=c$ };
\filldraw[black] (0,-0.6) circle (0pt) node[anchor=south]{$c_{N-1}=c$ };
\filldraw[black] (0,-1.1) circle (0pt) node[anchor=south]{$c_N=r$ };

\draw[dotted] (-2,0) -- (-1.5,1); 
\draw[dotted] (-2,0) -- (-1.5,0.5); 
\draw[dotted] (-2,0) -- (-1.5,-1); 
\draw[dotted] (-2,0) -- (-1.5,-0.5); 

\draw[dotted] (2,0) -- (1.5,1); 
\draw[dotted] (2,0) -- (1.5,0.5); 
\draw[dotted] (2,0) -- (1.5,-1); 
\draw[dotted] (2,0) -- (1.5,-0.5);

\draw[->] (-1.5,1)--(1.5,1) ;
\draw[->] (-1.5,0.5)--(1.5,0.5) ;
\draw[->] (-1.5,-0.5)--(1.5,-0.5) ;
\draw[->] (-1.5,-1)--(1.5,-1) ;

\end{tikzpicture}
\label{fig: single point to single point}
\caption{Transport multi-path components from $1$ point to $1$ point.}
\end{figure}

\begin{proof}  
Since the minimum path between two points in $\mathbb{R}^m$ is a line segment, by the optimality of $\vec{T}^*$, each $T_k = c_k \llbracket \overline{xy} \rrbracket$ for some $0 < c_k \le c$. On the other hand, as the only source point, $x\in \bigcap_{k=1}^{N} X(\eta_{k})$.
By Corollary \ref{cor: intersection capacity} , for $k=1,2,\cdots,N$, there is at most one component $T_k$ (without loss of generality assume this component is $T_N$) such that $T_N= r\llbracket \overline{xy} \rrbracket, r \in (0,c]$, and any other components are $T_k = c\llbracket \overline{xy} \rrbracket$.
The total number of components required is $N=\lceil m_1/c \rceil$, and since there is only one component that has mass less or equal to $c$, then $r = m_1 - (N-1)c$.
\end{proof}
\end{example}

\begin{example}
Suppose $\mu^- = m_1 \delta_{x_1} + m_2 \delta_{x_2},\ \mu^+ = (m_1 + m_2) \delta_y$, and $\vec{T}^* = (T_1,T_2,\cdots,T_N) \in Path_c(\mu^-,\mu^+)$ is optimal with $N\ge 2$.
Then, there exists an integer $1 \le n \leq N-1$ such that  up to a permutation of component indices, $$T_1= T_2=\cdots=T_{n-1} = c\llbracket \overline{x_1y} \rrbracket,\ T_{n +2}=T_{n +3}=\cdots=T_{N} = c\llbracket \overline{x_2y} \rrbracket$$ and 
one of the following cases hold:
\begin{itemize}
\item [Case 1:]
$ T_{n} = \theta_1\llbracket \overline{x_1y} \rrbracket, \ 
T_{n+1} =\theta_2 \llbracket \overline{x_2y} \rrbracket,$
where 
\begin{equation}
\label{eqn: n_theta}
n=\lceil m_1/c \rceil, \ \theta_1 = m_1 - (n-1)c \ \in (0, c] \text{ and } \theta_2 = m_2 - (N-n-1) c \ \in (0, c].
\end{equation}

\item [Case 2:]
$T_{n} = \theta_1\llbracket \overline{x_1y} \rrbracket \,  \text{ for some } \theta_1\in (0,c]  \text{ and }
T_{n+1} \text{ is Y-shaped}.$

\item [Case 3:]
$T_{n} \text{ is Y-shaped, and } T_{n+1}= \theta_2\llbracket \overline{x_2y} \rrbracket \text{ for some } \theta_2\in (0,c)$.
\end{itemize}

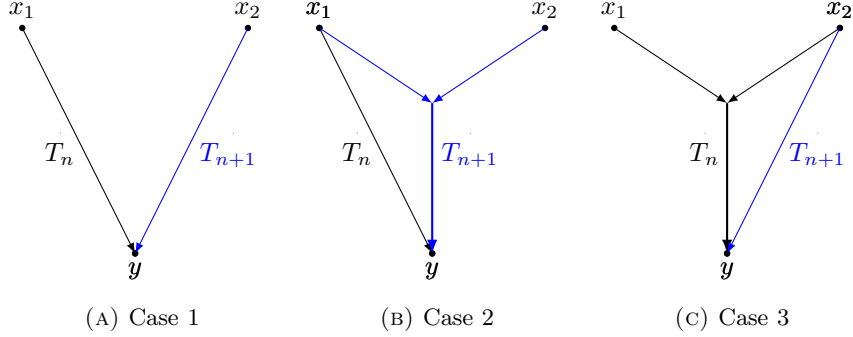
\begin{figure}[h]
\begin{subfigure}{0.3\textwidth}
\centering
\begin{tikzpicture}[>=latex]
\filldraw[black] (0,0) circle (1pt) node[anchor=north]{$y$};
\filldraw[black] (-1.5,3) circle (1pt) node[anchor=south]{$x_1$};

\filldraw[black] (-1,1.6) circle (0pt) node[anchor=north]{$T_{n}$ };
\draw[->] (-1.5,3) -- (0,0);

\filldraw[black] (0,0) circle (1pt) node[anchor=north]{$y$};
\filldraw[black] (1.5,3) circle (1pt) node[anchor=south]{$x_2$};

\filldraw[black] (1.3,1.6) circle (0pt) node[anchor=north]{\textcolor{blue}{$T_{n+1}$  } };
\draw[blue,->] (1.5,3) -- (0,0);

\end{tikzpicture}
\label{fig: Case 1}
\caption{Case 1}
\end{subfigure}
\begin{subfigure}{0.3\textwidth}
\centering
\begin{tikzpicture}[>=latex]
\filldraw[black] (0,0) circle (1pt) node[anchor=north]{$y$};
\filldraw[black] (-1.5,3) circle (1pt) node[anchor=south]{$x_1$};

\filldraw[black] (-1,1.6) circle (0pt) node[anchor=north]{$T_n$ };
\draw[->] (-1.5,3) -- (0,0);

\filldraw[black] (0,0) circle (1pt) node[anchor=north]{$y$};
\filldraw[black] (-1.5,3) circle (1pt) node[anchor=south]{$x_1$};
\filldraw[black] (1.5,3) circle (1pt) node[anchor=south]{$x_2$};

\filldraw[black] (0.5,1.6) circle (0pt) node[anchor=north]{\textcolor{blue}{$T_{n+1}$} };

\draw[blue,->] (-1.5,3) -- (0,2);
\draw[blue,->] ( 1.5,3) -- (0,2);
\draw[blue,thick,->]  (0,2)--(0,0);

\end{tikzpicture}
\label{fig: Case 2}
\caption{Case 2}
\end{subfigure}
\begin{subfigure}{0.3\textwidth}
\centering
\begin{tikzpicture}[>=latex]
\filldraw[black] (0,0) circle (1pt) node[anchor=north]{$y$};
\filldraw[black] (-1.5,3) circle (1pt) node[anchor=south]{$x_1$};
\filldraw[black] (1.5,3) circle (1pt) node[anchor=south]{$x_2$};

\draw[->] (-1.5,3) -- (0,2);
\draw[->] ( 1.5,3) -- (0,2);
\draw[thick,->]  (0,2)--(0,0);

\filldraw[black] (-0.3,1.6) circle (0pt) node[anchor=north]{$T_{n}$ };

\filldraw[black] (0,0) circle (1pt) node[anchor=north]{$y$};
\filldraw[black] (1.5,3) circle (1pt) node[anchor=south]{$x_2$};

\filldraw[black] (1.2,1.6) circle (0pt) node[anchor=north]{\textcolor{blue}{$T_{n+1}$}  };
\draw[blue,->] (1.5,3) -- (0,0);

\end{tikzpicture}
\label{fig: Case 3 }
\caption{Case 3}
\end{subfigure}

\label{fig: different cases of 2 points to 1 point }
\caption{Different cases of transport multi-paths from $2$ source points to $1$ target point.}
\end{figure}

\begin{proof}
We first observe that there is at most one $k=1,2,\cdots, N$ such that $|X(\eta_k)|=2$. Indeed, 
suppose  $|X(\eta_{k_1})| = |X(\eta_{k_2})|  =2$ 
for some $k_1,k_2 \in \{1,2,\cdots, N\}$ with $k_1 \neq k_2$.
Since $|supp (\mu^-)| = |\{x_1, x_2\}| =2$, it follows that  $X(\eta_{k_1}) = \{x_1, x_2\}=X(\eta_{k_2})$, 
which contradicts the result $|X(\eta_{k_1})\cap X(\eta_{k_2})|  \le 1$  given in equation  (\ref{eq:supple1}). As a result, there is at most one $k$ such that $|X(\eta_k)|=2$. 
In other words, among all $N$ components of $\vec{T}$, there is at most one ``Y-shaped" path and all others are straight line segments. 
Up to a permutation of indices, we may list components from $x_1$ to $y$ first, then the ``Y-shaped" one if any, and then the line segments from $x_2$ to $y$. 
Moreover, by Proposition $\ref{eq:prop_suppeq1}$,  with at most one exception, the density on each line segment from $x_i$ to $y$ reaches its maximum capacity $c$ for each $i=1,2$.

\begin{itemize}
    \item[Case 1:] When $\vec{T}$ has no ``Y-shaped" component, let $n, \theta_1, \theta_2$ be as in (\ref{eqn: n_theta}). Then,  up to a permutation of indexes, we have
    $$T_1= T_2=\cdots=T_{n-1} = c\llbracket \overline{x_1y} \rrbracket,\ T_{n} = \theta_1\llbracket \overline{x_1y} \rrbracket, \ 
T_{n+1} =\theta_2 \llbracket \overline{x_2y} \rrbracket, $$
and 
$$ T_{n +2}=T_{n +3}=\cdots=T_{N} = c\llbracket \overline{x_2y} \rrbracket.$$
\item[Case 2:] When $\vec{T}$ has one ``Y-shaped" component and all components supported on the line segment from $x_2$ to $y$ reach their maximum capacity $c$, up to a permutation of indexes, we may list components of $\vec{T}$ as
 $$T_1= T_2=\cdots=T_{n-1} = c\llbracket \overline{x_1y} \rrbracket,\ T_{n} = \theta_1\llbracket \overline{x_1y} \rrbracket, \ 
T_{n+1} \text{ is ``Y-shaped"},$$ 
and
$$\ T_{n +2}=T_{n +3}=\cdots=T_{N} = c\llbracket \overline{x_2y} \rrbracket,$$
where $\theta_1\in (0, c]$.
\item[Case 3:] When $\vec{T}$ has one ``Y-shaped" component and there exists one component supported on the line segment from $x_2$ to $y$ that does not reach its maximum capacity $c$, up to a permutation of indexes, we may list components of $\vec{T}$ as
$$T_1= T_2=\cdots=T_{n-2} = c\llbracket \overline{x_1y} \rrbracket,\ T_{n-1} = \theta_1\llbracket \overline{x_1y} \rrbracket, \ 
T_{n} \text{ is ``Y-shaped"},$$
and
$$ T_{n+1} = \theta_2\llbracket \overline{x_2y} \rrbracket, \ 
T_{n +2}=T_{n +3}=\cdots=T_{N} = c\llbracket \overline{x_2y} \rrbracket,$$
where $\theta_1\in (0, c]$ and $\theta_2\in (0, c)$. We will show that $\theta_1=c$.
For the sake of contradiction, assume $0<\theta_1<c$. 
Since $T_n$ is ``Y-shaped", we may write 
\[ T_n=a_1 I_{\gamma_{x_1,y}}+ a_2 I_{\gamma_{x_2,y}}, \]
where $a_i>0$ and $\gamma_{x_i,y}$
is a polyhedral curve from $x_i$ to $y$ for each $i=1,2$.

Let
$\epsilon_0 = \min\{\theta_1, \theta_2, c-\theta_1, c-\theta_2, a_1,a_2\} >0$, 
and define
$$\vec{S} = (S_1,S_2,\cdots,S_{N}),$$ where
$$S_{n-1} := \epsilon_0 \llbracket \overline{x_1y} \rrbracket,\ 
S_{n} := \epsilon_0 \left(I_{\gamma_{x_2,y}}- I_{\gamma_{x_1,y}}\right) ,\ 
S_{n+1} := -\epsilon_0 \llbracket \overline{x_2y} \rrbracket,
$$ 
and $S_k = 0 \mbox{ for } k \not= n-1, n , n+1$.
By construction, $S_k$ is on $T_k$ for each $k$.
The corresponding $\rho_k(x)$'s, where $\partial S_k = \rho_k(x) \partial T_k$ for each $k$, are
$$\begin{array}{c|cccc}
     & x=x_1 & x=x_2 & x=y & \textrm{otherwise} \\
     \hline
\rho_{n-1}(x) & \epsilon_0/\theta_1 & 0 & \epsilon_0/\theta_1 & 0 \\
\rho_{n}(x)& -\epsilon_0/a_1&\epsilon_0/a_2& 0 & 0 \\
\rho_{n+1}(x) & 0 & -\epsilon_0/\theta_2 & -\epsilon_0/\theta_2 & 0 \\ 
\end{array}$$
and $\rho_k(x)=0$, for $k \not= n-1,n,n+1$.  Direct calculation shows that  $\vec{S}$ constructed above satisfies the conditions in Theorem \ref{eq:thm_perturbation_capacity}. 
Since $\vec{T}$ is optimal and $\alpha \in (0,1)$, Theorem \ref{eq:thm_perturbation_capacity} gives
$\vec{S}=\vec{0}$. This contradicts with the non-vanishing $\vec{S}$ constructed above.

\end{itemize}
\end{proof}
\end{example}

\end{document}